\documentclass[12pt,twoside,english,reqno,a4paper]{amsart}

\usepackage[utf8]{inputenc}
\usepackage[english]{babel}
\usepackage[lined,boxed]{algorithm2e}
\usepackage{float} %needed to fix the position of a figure
\usepackage{hyperref}
\usepackage{fullpage}
\usepackage{cite} %https://tex.stackexchange.com/questions/160552/bibtex-order-references-as-in-bibtex-file
\usepackage{listings,graphicx,caption,mwe,amsmath,varioref,amscd,amssymb,color,xcolor,bm,amsthm,amsfonts,graphics,lineno,float,comment,soul}
\usepackage{mathrsfs} % mathscr F
\usepackage[foot]{amsaddr}
\usepackage{bm}
\usepackage{tcolorbox}
\usepackage{enumitem}
\usepackage{indentfirst}
\usepackage{array}
\allowdisplaybreaks

\newcommand{\ignore}[1]{}

\theoremstyle{plain}
\newtheorem{definition}{Definition}[section]
\newtheorem{theorem}[definition]{Theorem}
\newtheorem{proposition}[definition]{Proposition}
\newtheorem{lemma}[definition]{Lemma}

\newtheorem{remark}[definition]{Remark}

\numberwithin{equation}{section}

\def\dis{\displaystyle}
\DeclareMathOperator*{\supp}{supp}
\def\R{\mathbb{R}}
\def\Rd{\mathbb{R}^d}
\def\N{\mathbb{N}}
\def\norma#1#2{\|#1\|_{\lower 4pt \hbox{$ \scriptstyle #2$ }}}

\newcommand{\li}{\mathcal L}
\newcommand{\dd}{\tfrac{\mathrm{d}}{\mathrm{d}t}}
\newcommand{\di}{\mathrm{d}}

\newcommand{\uo}{\mathfrak{u}} % controllo ottimale
\newcommand{\ro}{\mathfrak{r}} % costato riscalato ottimale

\newcommand{\xo}{\mathbf{x}} % stato ottimale

\newcommand{\uv}{\bm{u}} % controllo qualsiasi

\newcommand{\xv}{\bm{x}} % stato qualsiasi
\newcommand{\rv}{\bm{r}} % costato qualsiasi

\newcommand{\muo}{\pmb{\mu}} % misura stato controllo
\newcommand{\nuo}{\pmb{\nu}} % probabilità stato costato
\newcommand{\rhoo}{\pmb{\rho}} % misura stato costato controllo
\newcommand{\sigo}{\pmb{\sigma}}

\newcommand{\K}{\mathcal{K}}
\newcommand{\cP}{\mathcal{P}}

\newcommand{\wvo}{\mathbf{w}} % funzione limite controlli ottimale
\newcommand{\PCR}{\mathcal{P}_c(\mathbb{R}^d)}

\newcommand{\PunoR}{\mathcal P_1(\mathbb{R}^d)}

\newcommand{\de}{\mathrm{d}}

\newcommand{\HI}{$\bm{(\mathrm{HI}}\bm{)}$}
\newcommand{\Hv}{$\bm{(\mathrm{H}}v\bm{)}$}
\newcommand{\Hh}{$\bm{(\mathrm{H}}h\bm{)}$}
\newcommand{\HL}{$\bm{(\mathrm{H}}L\bm{)}$}
\newcommand{\Hphi}{$\bm{(\mathrm{H}}\phi\bm{)}$}

\makeatletter
\@namedef{subjclassname@2020}{%
  \textup{2020} Mathematics Subject Classification}
\makeatother

\DeclareMathOperator*{\argmax}{argmax}

\author[S. Almi]{Stefano Almi}
\address[Stefano Almi]{Dipartimento di Matematica e Applicazioni ``R.~Caccioppoli'', Universit\`a di Napoli Federico II, via Cintia, 80126 Napoli,
Italy}\email{stefano.almi@unina.it}

\author[R. Durastanti]{Riccardo Durastanti}
\address[Riccardo Durastanti]{Dipartimento di Matematica e Applicazioni ``R.~Caccioppoli'', Universit\`a di Napoli Federico II, via Cintia, 80126 Napoli,
Italy}\email{riccardo.durastanti@unina.it}

\author[F. Solombrino]{Francesco Solombrino}
\address[Francesco Solombrino]{Dipartimento di Scienze e Tecnologie Biologiche ed Ambientali, Università del Salento, via Lecce-Monteroni, 73047 Lecce, Italy} \email{francesco.solombrino@unisalento.it}


\begin{document}
\title[Optimality condition under low regularity of controls]{Mean field first order optimality condition under low regularity of controls}

\begin{abstract}
We show that mean field optimal controls satisfy a first order optimality condition (at a.e.~time) without any a priori requirement on their spatial regularity. This principle is obtained by a careful limit procedure of the Pontryagin maximum principle for finite particle systems. In particular, our result applies to the case of mean field selective optimal control problems for multipopulation and replicator dynamics.
\end{abstract}

\maketitle
\tableofcontents

\section{Introduction}
\label{s:intro}

The Pontryagin Maximum Principle (PMP) has been a cornerstone in the optimal control theory, providing necessary conditions for the optimality of control trajectories in single-agent systems (see, e.g.,~\cite{Bressan}). It converts an (integral) optimal control problem into a pointwise maximization of an Hamiltonian function. The latter also drives the forward-backward flow of optimal trajectories in the product space of states and co-states, respectively. As the complexity of systems increases, particularly in the context of multiagent systems, both the control problem and the traditional PMP framework become less tractable. In recent years, the Mean Field Pontryagin Maximum Principle (MF-PMP) has emerged as a powerful tool for addressing optimal control in large-scale systems, characterized by a large number of interacting agents. We refer to~\cite{ADS, BFRS, BonFor, BonRos, PMPWassConst, BonFra, Burger1, Burger2} for a (non-exhaustive) list of references on the topic.

The mean field optimal control problem one aims to solve is usually formulated as
\begin{align}
\label{e:controlpb-intro}
\min_{\boldsymbol{w}} \,  \int_{0}^{T} \int_{ \R^{d} } L( \mu_{t} ) \, \di \mu_{t} \, \di t + \int_{0}^{T} \int_{ \R^{d} } \phi (\boldsymbol{w} (t, x)) \, \di \mu_{t} \, \di t\,,
\end{align}
subject to the following continuity equation
\begin{align}
\label{e:continuity-intro}
\partial_{t} \mu_{t} + {\rm div} \big( ( \boldsymbol{v} + \boldsymbol{w} ) \mu_{t} \big) = 0\,,
\end{align}
describing the evolution of the agents' distribution $\mu_{t} \in \mathcal{P} (\R^{d})$, driven by a velocity field $\boldsymbol{v}$ and the additional control drift $\boldsymbol{w}$. This can be seen as a limit control problem as the number of agents tends to infinity. A rigorous derivation of~\eqref{e:controlpb-intro}--\eqref{e:continuity-intro} from a finite-particle control problem was obtained via $\Gamma$-convergence in~\cite{CLOS, FLOS, AAMS}. In all these contributions, the superposition principle (cf.~\cite{AGS, AFMS, Amb-Tre, Smirnov, Lisini}) is a key {\em trait-d'union} between the discrete and the continuous state equation, as the competitors in the limit problem are recovered as a suitable limit (in a measure theoretical sense) of discrete controls. As no {\em a priori} regularity constraints are imposed for the finite-particle problem, mean-field optimal controls may not enjoy any continuity property with respect to the space variable~$x$. It is also well-known (see \cite{FLOS}) that in general one must expect the optimal controls obtained by this procedure to be indeed closed-loop.

When coming to the derivation of first-order optimality conditions for problem~\eqref{e:controlpb-intro}--\eqref{e:continuity-intro}, the possible lack of regularity creates a gap with the available results provided by recent literature~\cite{BonRos, BonFra, Burger1}. In such contributions, as it will also happen in our paper, optimality conditions are formulated by coupling:
{\em (i)} a maximality condition on the Hamiltonian function in the space of admissible controls to be satisfied at (almost) every time $t$; {\em (ii)} a forward-backward flow in the product space of states and co-states, usually seen as probability measures on the phase space. 

Such formulation requires the development of local differentiability notions in Wasserstein spaces (see~\cite{BonFra, BonRos, ADS}). The proofs furthermore make use of an infinite-dimensional generalization of the classical {\em needle-variation argument}, which strongly relies on a Cauchy-Lipschitz theory for continuity equations of the form~\eqref{e:continuity-intro}. Indeed, while the equation~\eqref{e:continuity-intro} makes sense whenever the velocity~$\boldsymbol{v}$ and control~$\boldsymbol{w}$ are measurable and satisfy some integrability bounds (cf.~\cite[Chapter~8]{AGS}), uniqueness and stability properties are the outcome of additional regularity properties on the control drift field. Furthermore, the formulation of the adjoint equation for the case of closed-loop controls features spatial derivatives of the control field. Therefore, such formulations are in general not feasible for mean-field optimal control problems. On the one hand, the aim of a control law designed for the kinetic model is to provide a strategy which can be in turn applied -- either exactly or approximately -- to the corresponding finite-dimensional systems. On the other hand, the MF-PMP requires additional regularity constraints, which are rather artificial for optimal controls, restricting their range of applicability. A remarkable result in this sense is the one in~\cite{LipReg}, where uniform Lipschitz bounds along sequences of approximations by empirical measures are obtained at the price of some reasonable, though rather strong assumptions on the Lagrangian and the control cost of the problem.

The aim of the present paper is instead to recover a first-order necessary condition which runs in parallel to the variational limit procedure from the discretized to the kinetic problem,  taking into account the possible lack of regularity of minimizers in the continuum setting. Such condition is obtained as limit of a discretized PMP, under essentially the same assumptions that guarantee the variational convergence to a mean-field control problem, in the sense of \cite{CLOS, FLOS, AAMS}. This generalizes a similar point of view taken in \cite{BFRS}, where, however, a discrete subset of leaders was fixed from the beginning, and the control laws were open-loop and only acting on the leaders' population. The condition we recover is similar to the one obtained via needle variation, as it involves the pointwise maximization of an Hamiltonian functional coupled to an evolution equation for state and co-state variables. The main difference lays in the absence of spatial derivatives of the optimal control field, which always appear in the needle variation approach. Similar to~\cite{BFRS}, the derivation of optimality conditions relies on some continuity properties of Wasserstein differentials with respect to the convergence at hand. A crucial result of our analysis, which allows us to overcome the discrete setting of~\cite{BFRS}, is a representation result for the limit of the discrete control measures associated to the minimizers and the corresponding trajectories. We namely show (see Lemma \ref{lemma3}) that the Radon-Nykodim derivative of the limit control density (denoted by $\rhoo_t$ in the statement) with respect to the density $\nuo_t$ in the product space of positions~$x$ and co-states~$r$ is independent of $r$ and agrees with the optimal control~$w(t,x)$. The proof makes use of the Disintegration theorem, of a semicontinuity result for superlinear convex functionals on measures which was also instrumental to the results in \cite{FLOS} and \cite{AAMS}, and only requires strict convexity of the control cost $\phi$ and fair differentiability assumptions on the Lagrangian cost~$L$.

The main result of our paper (cf.~Theorem \ref{mainres}) also encompasses some useful additional features in modeling. For instance, as in \cite{AAMS}, we can allow the policy maker for a {\it selective} type of control in a system of the type
\begin{displaymath}
\partial_{t} \mu_{t} + {\rm div} \big( ( \boldsymbol{v} + h \boldsymbol{w} ) \mu_{t} \big) = 0\,,
\end{displaymath}
where $h\geq0$ is a non-negative activation function selecting the set of agents targeted by the decision of the policy maker, depending on their state and, possibly, on the global state of the system. Furthermore, as we show in Section \ref{s:convexcase}, the results can be extended to multi-populations setting with time-evolving labels according to a distribution $\lambda \in \mathcal{P}(U)$ which may account, for instance, for a different degree of influence of the single agent, as in~\cite{MS2020, AAMS, AFMS, Mor1, BFS, During, Toscani}. For this, one needs to consider suitable notions of differentiability for functions defined on convex subsets of Banach spaces and on measure defined on these convex subsets, introduced in \cite{AFMS} and \cite{ADS} respectively.

\section{Preliminaries and notation}
\label{s:preliminaries}

We consider a separable Radon metric space $(X, d)$. When $X=\Rd$ ($d\geq 1$) we adopt the distance induced by the Euclidean norm $|\cdot|$. For $R>0$ we define $B_R(x): = \{ \tilde x\in X : d(\tilde x,x)\leq R \}$. In the Euclidean setting we define $B_R:=B_R(0)$ and by $\langle \cdot, \cdot\rangle$ we denote the Euclidean scalar product. For a vector $\bm{x}^N\in (\R^d)^N$ we indicate with $x_i\in \R^d$ its $i$-th component (for initial data $\mathbf{x}^N_0\in (\R^d)^N$ we use $\mathbf{x}^N_{0,i}\in \R^d$ respectively). For a vector $v\in \R^d$ we indicate with $v^i\in \R$ its $i$-th component and for any pair of vectors $v,w\in \R^d$ we represent with $v\otimes w\in \R^{d\times d}$ the matrix with component $v^i w^j$ at the $i$-th row and $j$-th column for every $i,j=1,\dots,d$. We denote by $\mathcal{L}$ the Lebesgue measure on~$\R$. \\
We denote by $\mathcal{M}(X)$ the space of Borel measures with bounded total variation and by $\mathcal{P}(X)$ the family of all Borel probability measures on~$X$. For $p \geq 1$ we
further consider
$$
\mathcal{P}_p(X):=\left\{\mu \in \mathcal{P}(X) : \int_X d(x,\bar{x})^p \di\mu(x)<+\infty \text{ for some }\bar{x}\in X  \right\}
$$
and $\mathcal{P}_c(X)$ the subset of $\mathcal{P}(X)$ of measures with compact support in $X$ recalling that the support is the closed set
$$
\supp(\mu)=\{x\in X : \mu(V)>0 \text{ for each neighborhood }V\text{ of }x  \}.
$$
If $X$ is contained in some Banach space $Z$, we define the $p$ momentum of $\mu\in\mathcal{P}(X)$ as
$$
m_p(\mu):=\left(\int_X \|x\|_Z^p \,\di \mu(x)\right)^\frac1p \qquad \mbox{for }p\geq 1.
$$
Let $X_1$ and $X_2$ be separable Radon metric spaces, we define for every $\mu_1\in \mathcal{P}(X_1)$ and $\mu_2\in\mathcal{P}(X_2)$ the transport plans
with marginals $\mu_1$ and $\mu_2$
$$
\Gamma(\mu_1,\mu_2):=\left\{{\bm\gamma}\in \mathcal{P}(X_1\times X_2) : \pi_{\#}^{i}{\bm\gamma}=\mu_i \text{ for }i=1,2 \right\},
$$
where $\pi^i \colon X_1\times X_2 \to X_i$ is the projection on $X_i$ and $\pi_{\#}^{i}{\bm\gamma}\in \mathcal{P}(X_i)$ is the pushforward of
${\bm\gamma}$ through $\pi^i$. Note that $\Gamma(\mu_1,\mu_2)$ is a non-empty and compact subset of $\mathcal{P}(X_1\times X_2)$ (see Remark 5.2.3 of \cite{AGS}).
We define the $p$-Wasserstein distance between two probability measures $\mu_1$ and $\mu_2$ in $\mathcal{P}_p(X)$ by
$$
W_p^p(\mu_1,\mu_2)= \min\left\{\int_{X\times X}d(x_1,x_2)^p \di{\bm\gamma}(x_1,x_2) : {\bm\gamma}\in \Gamma(\mu_1,\mu_2)\right\}\,.
$$
It follows from \cite[ Proposition 7.15]{AGS} that $\mathcal{P}_p(X)$ endowed with the $p$-Wasserstein distance is a separable metric space which is
complete if $X$ is complete.
We define
$$
\mathcal{P}_c(X):=\left\{\mu \in \mathcal{P}(X) : \supp(\mu) \text{ is compact in }X \right\}
$$
and
$$
\mathcal{P}_b(X):=\left\{\mu \in \mathcal{P}(X) : \supp(\mu) \text{ is bounded in }X \right\}.
$$
From now on, unless otherwise specified, when we write $\mathcal{P}_p(X)$, $\mathcal{P}_c(X)$ and $\mathcal{P}_b(X)$ we mean the separable metric spaces $\left(\mathcal{P}_p(X), W_p\right)$, $\left(\mathcal{P}_c(X), W_1\right)$ and $\left(\mathcal{P}_b(X), W_1\right)$ respectively. Moreover, in $\mathcal{P}_1(X)$, we consider the subset $\mathcal{P}^N(X)$ of discrete measures
$$
\mathcal{P}^N(X):=\left\{\mu\in\mathcal{P}_1(X) : \exists (x_1,\dots,x_N)\in X^N \text{ such that } \mu=\frac{1}{N}\sum_{i=1}^N \delta_{x_i} \right\} \subset \mathcal{P}_b(X).
$$
We define the \emph{R-fattening} of the support of a measure $\mu\in\mathcal{P}_c(X)$ as
$$
B_R^\mu:= \displaystyle\bigcup_{x\in\supp(\mu)} B_{R}(x).
$$
Note that, since $\mu$ has compact support, if $\eta\in\mathcal{P}(B_R^\mu)$ then $\eta \in \mathcal{P}_b(X)$. Moreover, if $X$ is a proper space (for instance a compact space or an Euclidean space), then $\mathcal{P}_b(X)$ coincides with $\mathcal{P}_c(X)$. \\

In view of the applications of our main result to the replicator dynamics (see Section \ref{s:convexcase}), we give the definition of differentiability with respect to a probability measure in the context of convex subspaces of Banach spaces. 
From now on let $(E,\|\cdot\|_E)$ be a separable Banach space, let $C$ be a closed and convex subset of $E$ and let $E_C$ be the topological closure of the vector subspace $\{\alpha(c_1-c_2) : \alpha\in \R, c_1,c_2\in C\}$. Let $A\colon \mathcal{P}_1(C)\to Y$ be such that $\mathcal{P}_b(C) \subseteq D(A):=\{\psi\in\mathcal{P}_1(C) : \|A(\psi)\|_Y<+\infty\}$, with $(Y,\|\cdot\|_Y)$ a Banach space. We introduce a definition of (strong) local differentiability of $A$ at $\mu\in \mathcal{P}_c(C)$ (see \cite[Definition 4.2]{ADS} and  \cite[Section 2]{BPLL}).

\begin{definition}
\label{Wmudiff}
A functional $A \colon \mathcal{P}_{1} (C) \to Y$ is (strongly) locally differentiable at $\mu\in \mathcal{P}_c(C)$ if there exists a map $\nabla_\psi A(\mu)\in L^2_{\mu}(C; \li(E_C;Y))$ in Bochner sense such that for every $R>0$ and for every $\nu\in \mathcal{P}(B^\mu_R)$ it holds 
$$
A(\nu)-A(\mu)=\int_{C\times C} \nabla_\psi A(\mu)(c_1)[c_2-c_1] \di{\bm\gamma}(c_1,c_2) + o_R(W_{2,{\bm\gamma}}(\mu,\nu))
$$
for any ${\bm\gamma}\in \Gamma(\mu,\nu)$, where  $\lim_{t \searrow 0} \, \frac{o_{R} (t) }{t} = 0$ and
\begin{align*}
W_{2,{\bm\gamma}}^2(\mu,\nu):=\int_{C\times C} \|c_1-c_2\|_E^2 \di{\bm\gamma}(c_1,c_2)\,,\\
\end{align*}
\end{definition}

In what follows we deal with systems of $N$ agents. Such agents can be identified by a vector $(x_1,\dots, x_N)\in X^N$. In particular, up to permutation, every $N$-tuple $\bm{x}=(x_1,\dots,x_N)$ can be represented with $\psi^N=\frac{1}{N}\sum_{i=1}^N \delta_{x_i}\in \mathcal{P}^N(X)$. This is the meaning whenever we say that the vector $\bm{x}$ has a generated measure $\psi^N$. We give the following useful definition.

\begin{definition}
\label{sym-def}
Let $F\colon X\times X^N \to Y$ with $X,Y$ metric spaces. We say that $F$ is symmetric if $F(x,\sigma(\bm x))= F(x,\bm x)$ for every $x\in X$, every $\bm x\in X^N$ and every permutation $\sigma: X^N \to X^N$.
\end{definition}

It follows that, if $F:X\times X^N \to Y$ is symmetric, we can uniquely identify $F(x, \bm x)$ with $F(x, \psi^N)$ (setting $F(x,\psi^N):=F(x, \bm x)$) and, as a consequence, we can consider $F$ defined on $X \times \mathcal{P}^N(X)$. On the other hand, if we have a map $F:X \times \mathcal{P}_1(X)\to Y$ we can always consider the restriction $F_{| X^N}: X\times X^N \to Y$ (setting $F_{| X^N}(x, \xv):= F(x,\psi^N)$) which for simplicity of notation we denote again with $F$ (except in the next lemma). Therefore this restriction is symmetric in the sense of Definition \ref{sym-def}. We will use this identification throughout the paper. \\
In the case $X=\R^d$ and $Y=\R^m$ ($m\geq 1$), we give a useful lemma which clarifies as the above identification links the strong local differential of a map $F$ at $\psi^N$ with the classical notion of differentiability at a point $\bm x\in (\Rd)^N$ (a related result is contained in \cite[Proposition 4]{LipReg}).
\begin{lemma}
\label{grad-identi}
Let $F:\R^d \times \mathcal{P}_1(\R^d)\to \R^m$ be locally differentiable at $\psi^N$ in the sense of Definition \ref{Wmudiff}, where $\psi^N$ is generated by $\xv$ (i.e. $\psi^N=\frac{1}{N}\sum_{i=1}^N \delta_{x_i}$). Then $F_{|(\Rd)^N }\colon \Rd\times (\Rd)^N \to \R^m$ is differentiable at $\bm x$ and it holds that
\begin{equation}
\label{ide1}
\nabla_{\psi}F(x,\psi^N)(x_i) = N\nabla_{x_i}F_{| (\Rd)^N}(x,\bm x) = \sum_{h=1}^N \nabla_{x_h}F_{|(\Rd)^N }(x,\bm x) \qquad \text{for every } i=1,\dots, N.
\end{equation}
Moreover, if, for every fixed $x\in \R^d$, the map $\mathcal{P}_c(\R^d) \times \R^d \ni (\psi,\tilde x) \mapsto \nabla_\psi F(x,\psi)(\tilde x)\in \R^{d\times d}$ is continuous, then $F_{| (\Rd)^N}(x,\cdot) \in C^1((\Rd)^N;\R^m)$. 
\end{lemma}

\begin{proof}
Let $\psi^N$ and $\tilde \psi^N$ be the empirical measures generated by $\xv =(x_1,\dots,x_N)\in (\R^d)^N$ and $\tilde\xv=(\tilde x_1,\dots,\tilde x_N)\in (\R^d)^N$ respectively. Then, by Definition \ref{Wmudiff} choosing $\bm\gamma = \frac{1}{N}\sum_{j=1}^N\delta_{(x_j,\tilde x_j)} \in \Gamma(\psi^N, \tilde\psi^N)$, and by the identification of $F(x,\psi^N)$ with a symmetric $F_{|(\Rd)^N}(x,\xv)$ (see Definition \ref{sym-def} and the related comment), we have
\begin{eqnarray*}
&&F_{|(\Rd)^N}(x,\tilde \xv) - F_{|(\Rd)^N}(x,\xv) =  F(x,\tilde \psi^N)-F(x,\psi^N) \nonumber \\
&&=\int_{\R^{2d}} \nabla_{\psi}F(x,\psi^N)(x')[x''-x'] \di\bm\gamma(x',x'') + o_R\left(\left(\int_{\R^{2d}} |x'-x''|^2 \di{\bm\gamma}(x',x'')\right)^{\frac{1}{2}}\right)\nonumber \\
&&= \frac{1}{N} \sum_{j=1}^N \nabla_{\psi}F(x,\psi^N)(x_j)[\tilde x_j-x_j] + \frac{1}{\sqrt N}o\left(\|\xv -\tilde\xv\|_{(\R^d)^N}\right),
\end{eqnarray*}
for $R>0$ fixed and sufficiently large. It follows from the previous equalities and by definition of differentiability in the Euclidean setting that $F_{|(\Rd)^N}$ is differentiable at $\bm x$ and that 
$$
\nabla_{x_i}F_{|(\Rd)^N}(x,\bm x) = \frac{1}{N}\nabla_{\psi}F(x,\psi^N)(x_i) \qquad \forall i=1,\dots,N,
$$
which gives the first equality in \eqref{ide1}. The second equality of \eqref{ide1} is a direct consequence of the fact that $F_{|(\Rd)^N}$ is symmetric (see Definition \ref{sym-def}). \\
Finally, using again that 
$$
W_{2,\bm\gamma}(\psi^N,\tilde\psi^N)= \frac{1}{\sqrt N}\|\xv -\tilde\xv\|_{(\R^d)^N},
$$
and by the identification \eqref{ide1} we deduce the last result of the lemma on the continuous differentiability of $F_{|(\Rd)^N}$. 
\end{proof}

Now we recall a result contained in \cite{FLOS} that will be used in what follows. Let $\phi \colon \R^d\to [0,+\infty]$ be a proper, lower semicontinuous, convex function superlinear at infinity and such that $\phi(0)=0$, let $\psi\in\mathcal{M}^+(X)$ be a reference measure and $\mu\in\mathcal{M}(X;\R^d)$ be a vector measure. We define the following functional:
\begin{equation}
\label{def-flos}
\Phi(\mu|\psi):=\int_{X} \phi(\omega(x))\di\psi(x) \quad \text{if }\mu=\omega\psi \ll \psi, \quad \Phi(\mu|\psi):=+\infty \quad \text{if }\mu \not\ll\psi.
\end{equation}
Then the following hold.
\begin{lemma}[{\cite[Theorem 2.6]{FLOS}}]
\label{flos1}
Suppose that we have two sequences $\psi^n\in\mathcal{M}^+(X)$, $\mu^n\in\mathcal{M}(X;\R^d)$ narrowly converging to $\psi\in\mathcal{M}^+(X)$ and $\mu\in\mathcal{M}(X;\R^d)$, respectively. Then 
$$
\liminf_{n\to +\infty} \Phi(\mu^n|\psi^n) \geq \Phi(\mu|\psi).
$$
In particular, if $\liminf_{n\to +\infty} \Phi(\mu^n|\psi^n)< +\infty$, we have $\mu\ll \psi$.
\end{lemma}

Finally, for $C$ closed and convex in a separable Banach space $E$, we say that $\psi\in C([0,T];\mathcal{P}_1(C))$ is a solution to a continuity equation 
$$
\begin{cases}
\dis \dd \psi_t = -\mathrm{div}_{x}\left(V(x,\psi_t)\psi_t\right) & \text{ in }(0,T],  \\
\psi_0=\hat{\psi}_0,
\end{cases}
$$
with $V\colon C\times \mathcal{P}_1(C)\to E$ if for every $\varphi\in C_c^\infty([0,T)\times E)$ and for every $t\in [0,T]$ it holds
$$
\int_E \varphi(t,x)\di\psi_t(x) - \int_E \varphi(0,x)\di\hat{\psi}_0(x) = \int_0^t \int_E \Big(\partial_t \varphi(\tau,x) + \langle \nabla_x \varphi(\tau,x), V(x,\psi_\tau) \rangle_{E^*\times E} \Big)\di\psi_\tau(x)\di\tau,
$$
where $\nabla_x$ is the Fr\'echet differential and the integrals are defined in Bochner sense.

\section{Assumptions and main results}
\label{s:mainres}

Throughout the work we assume the following on controls and initial data of the optimal control problems we will study.

\begin{tcolorbox}
$$\bm{(\mathrm{HI}}\bm{)}$$ Let $K$ be a compact and convex subset of $\Rd$ of admissible controls such that $0\in K$ and $\K:=L^1([0,T];K)$. Let $\hat{\Psi}_0\in \mathcal{P}_c(\R^d)$. Let $\xo_0^N \in (\R^d)^N$ and $\Psi_0^N:=\frac{1}{N}\sum_{i=1}^N \delta_{\xo_{0,i}^N}\in \mathcal{P}^N(\R^d)$ be such that $\supp(\Psi_0^N)\subseteq B_1^{\hat{\Psi}_0}$ for every $N\in \N$ and $\Psi_0^N \rightharpoonup \hat{\Psi}_0$ weakly* in the sense of measures as $N\to +\infty$.
\end{tcolorbox}

Note that such an approximation for $\hat{\Psi}_0$ is always possible, see, e.g., \cite[Section 3]{MS2020}. \\

For every $N\in \N$ we introduce the following particle optimal control problem: 
\begin{equation}
\label{costoN}
\min_{\bm u^N\in\K^N}\left\{\mathcal{F}_N^{\xo_0^N}(\bm x^N,\bm u^N):= \int_0^T L(\psi_t^N) \di t + \int_0^T \frac{1}{N}\sum_{i=1}^N \phi(u_i(t))\di t \right\}
\end{equation}
where $\bm x^N\in AC([0,T]; (\Rd)^N)$ is the solution to
\begin{equation}
\label{systN}
\begin{cases}
\dis \dd x_i(t)= v(x_i(t),\psi_t^N) + h(x_i(t),\psi_t^N)u_i(t) & \text{ in }(0,T],  \\
x_i (0)=\xo_{0,i}^N,
\end{cases}
\quad\text{for }i=1,\dots, N,
\end{equation}
and $\psi_t^N:= \frac{1}{N}\sum_{i=1}^N \delta_{x_i(t)}$,  thus identifying $\bm x^N\in AC([0,T]; (\Rd)^N)$ with $\psi^N\in AC([0,T];\mathcal{P}^N(\Rd))$. \\
We assume the following for the components $v$ and $h$ of the velocity field.
\begin{tcolorbox}
$$\bm{(\mathrm{H}}v\bm{)}$$ Let $v:\R^d \times \mathcal{P}_1(\R^d) \to \R^d$ be such that:
\begin{itemize}
\item[$(v_1)$] for every $R>0$ there exists $L_{v,R}>0$ such that for every $x_1,x_2\in B_R$ and every $\psi_1,\psi_2\in \mathcal{P}(B_R)$
$$
|v(x_1,\psi_1)-v(x_2,\psi_2)| \leq L_{v,R} \left(|x_1-x_2| + W_1(\psi_1,\psi_2)\right);
$$
\item[$(v_2)$] there exists $M_v>0$ such that for every $x\in\R^d$ and every $\psi\in \mathcal{P}_1(\R^d)$
$$
|v(x,\psi)|\leq M_v\left(1+|x|+m_1(\psi)\right);
$$
\item[$(v_3)$] for every $\psi\in \mathcal{P}_1(\R^d)$ the map $x\mapsto v(x,\psi)$ is differentiable with differential $\nabla_x v(x,\psi)\in\R^{d\times d}$ and the map $\R^d\times \mathcal{P}_1(\R^d) \ni (x,\psi)\mapsto \nabla_x v(x,\psi)\in \R^{d\times d}$ is continuous;
\item[$(v_4)$] for every $x\in \R^d$ the map $\psi\mapsto v(x,\psi)$ is locally differentiable w.r.t. Definition \ref{Wmudiff} with differential $\nabla_\psi v(x,\psi)$ and the map $\R^d \times \mathcal{P}_c(\R^d) \times \R^d \ni (x,\psi,\tilde x) \mapsto \nabla_\psi v(x,\psi)(\tilde x)\in \R^{d\times d}$ is continuous.
\end{itemize}
\end{tcolorbox}

\begin{tcolorbox}
$$\bm{(\mathrm{H}}h\bm{)}$$ Let $h:\R^d \times \mathcal{P}_1(\R^d) \to \R$ be such that:
\begin{itemize}
\item[$(h_1)$] $h$ is bounded uniformly with respect to $\psi\in \mathcal{P}_1(\R^d)$;
\item[$(h_2)$] for every $R>0$ there exists $L_{h,R}>0$ such that for every $x_1,x_2\in B_R$ and every $\psi_1,\psi_2\in \mathcal{P}(B_R)$
$$
|h(x_1,\psi_1)-h(x_2,\psi_2)| \leq L_{h,R} \left(|x_1-x_2| + W_1(\psi_1,\psi_2)\right);
$$
\item[$(h_3)$] for every $\psi\in \mathcal{P}_1(\R^d)$ the map $x\mapsto h(x,\psi)$ is differentiable with differential $\nabla_x h(x,\psi)\in\R^{d}$
   and the map $\R^d\times \mathcal{P}_1(\R^d) \ni (x,\psi)\mapsto \nabla_x h(x,\psi)\in \R^d$ is continuous;
\item[$(h_4)$] for every $x\in \R^d$ the map $\psi\mapsto h(x,\psi)$ is locally differentiable w.r.t. Definition \ref{Wmudiff} with differential $\nabla_\psi h(x,\psi)$ and the map $\R^d \times \mathcal{P}_c(\R^d) \times \R^d \ni (x,\psi,\tilde x) \mapsto \nabla_\psi h(x,\psi)(\tilde x)\in \R^d$ is continuous.
\end{itemize}
\end{tcolorbox}

For the cost functions $L$ and $\phi$ the following assumptions hold.

\begin{tcolorbox}
$$\bm{(\mathrm{H}}L\bm{)}$$ Let $L:\mathcal{P}_1(\R^d) \to [0,+\infty)$ be such that:
\begin{itemize}
\item[$(L_1)$] for every $R>0$ there exists $L_{L,R}>0$ such that for every $\psi_1,\psi_2\in \mathcal{P}(B_R)$
$$
|L(\psi_1)-L(\psi_2)| \leq L_{L,R} W_1(\psi_1,\psi_2);
$$
\item[$(L_2)$] $L$ is locally differentiable w.r.t. Definition \ref{Wmudiff} with differential $\nabla_\psi L(\psi)$ and the map $\mathcal{P}_c(\R^d) \times \R^d \ni (\psi,\tilde x) \mapsto \nabla_\psi L(\psi)(\tilde x)\in \R^d$ is continuous.
\end{itemize}
\end{tcolorbox}

\begin{tcolorbox}
$$\bm{(\mathrm{H}}\phi\bm{)}$$ Let $\phi:\R^d \to [0,+\infty)$ be strictly convex with $\phi(0)=0$.
\end{tcolorbox}

Under assumptions \HI, \Hv-$(v_1,v_2)$, \Hh-$(h_1,h_2)$, \HL-$(L_1)$ and \Hphi, by~\cite[Proposition 2]{AAMS}, there exists an optimal trajectory-control pair $(\xo^N,\uo^N)\in AC([0,T];(\Rd)^N) \times \K^N$ for \eqref{costoN}-\eqref{systN}. We define the generated pairs
\begin{equation}
\label{def-psi}
\Psi_t^N:= \frac{1}{N}\sum_{i=1}^N \delta_{\xo_i(t)}\in\PunoR, \qquad  \Psi^N:=\Psi_t^N \otimes \mathcal{L}_{|[0,T]}\in C([0,T];\mathcal{P}_1(\R^d)),
\end{equation}
and
\begin{equation}
\label{def-muo}
\muo_t^N:= \frac{1}{N}\sum_{i=1}^N \uo_i(t)\delta_{\xo_i(t)}\in \mathcal{M}(\R^d; \R^d), \qquad \muo^N:=\muo_t^N \otimes \mathcal{L}_{|[0,T]}\in \mathcal{M}([0,T]\times \R^d; \R^d).
\end{equation}

In addition, by \cite[Lemma 1 and Proposition 2]{AAMS} (which are an adaptation of \cite[Lemma 6.2]{FLOS}), we know the behavior of $\Phi$ (defined by \eqref{def-flos}) when it is evaluated on empirical measures and, in particular, on $(\Psi_t^N,\muo_t^N)$. More precisely, we have the following result.
\begin{lemma}
\label{flos2}
Assume $\bm{(\mathrm{H}}\phi\bm{)}$. Let $(\xv^N,\uv^N)\in AC([0,T];(\R^d)^N)\times \K^N$, and let $(\psi^N,\mu^N)\in AC([0,T];\mathcal{P}^N(\R^d)) \times \mathcal{M}([0,T]\times \R^d;\R^d)$ be the pair generated by $(\xv^N,\uv^N)$. Then, for a.e. $t\in[0,T]$ we have
$$
\frac{1}{N}\sum_{i=1}^N \phi(u_i(t)) \geq \Phi(\mu^N_t|\psi^N_t).
$$
Moreover, it holds for a.e. $t\in [0,T]$ that
$$
\frac{1}{N}\sum_{i=1}^N \phi(\uo_i(t)) = \Phi(\muo^N_t|\Psi^N_t),
$$
where $(\Psi_t^N,\muo_t^N)$ are defined by \eqref{def-psi} and \eqref{def-muo} respectively.
\end{lemma}

The limit as $N\to +\infty$ for the optimal control problem \eqref{costoN}-\eqref{systN} is established in \cite[Corollary 1]{AAMS} (see also \cite[Theorem 3.3]{FLOS}). In particular, the following result is proved which we rewrite in a way more suitable for our aim.
\begin{proposition}
\label{vecchio}
Assume \HI, \Hv-$(v_1,v_2)$, \Hh-$(h_1,h_2)$, \HL-$(L_1)$ and \Hphi. Then for every optimal trajectory-control pair $(\xo^N,\uo^N)\in AC([0,T];(\Rd)^N) \times \K^N$ for \eqref{costoN}-\eqref{systN} with generated pairs $(\Psi^N,\muo^N)$ there exists $(\Psi,\muo)$ such that, up to subsequence, $\Psi^N \to \Psi$ in $C([0,T];\PunoR)$ and $\muo^N \stackrel{*}\rightharpoonup \muo$ in $\mathcal{M}([0,T]\times \R^d;\R^d)$ with $\muo=\wvo\Psi$ for some $\wvo\in L^1_{\Psi}([0,T]\times \R^d; K)$. Moreover the pair $(\Psi,\wvo)$ is a solution to the optimal control problem
\begin{equation}
\label{costo}
\min_{\omega\in L^1_{\psi}([0,T]\times \R^d; K)}\left\{\mathcal{F}^{\hat{\Psi}_0}(\psi,\omega):= \int_0^T L(\psi_t) \di t + \int_0^T \int_{\R^d} \phi(\omega(t,x))\di \psi_t(x) \di t \right\}
\end{equation}
where $\psi \in C([0,T]; \PunoR)$ is a distributional solution to
\begin{equation}
\label{syst}
\begin{cases}
\dis \dd \psi_t = -\mathrm{div}_{x}\left((v(x,\psi_t) + h(x,\psi_t)\omega(t,x))\psi_t\right) & \text{ in }(0,T],  \\
\psi_0=\hat{\Psi}_0.
\end{cases}
\end{equation}
Finally, it holds that
\begin{equation}
\label{gamma-conv}
\mathcal{F}^{\hat{\Psi}_0}(\Psi,\wvo) = \lim_{N\to +\infty} \mathcal{F}_N^{\xo_0^N}(\xo^N,\uo^N).
\end{equation}
\end{proposition}

\begin{remark}
Let $\uo^N$, $\Psi$, $\muo$ and $\wvo$ be as in Proposition \ref{vecchio}. Observe that, by \eqref{gamma-conv}, and since $\Psi^N$ converges to $\Psi$ in $C([0,T];\PunoR)$, it follows that 
\begin{equation}
\label{gamma-conv2}
\lim_{N\to +\infty} \int_0^T \frac{1}{N}\sum_{i=1}^N\phi(\uo_i(t))\di t = \int_0^T \int_{\R^d} \phi(\wvo(t,x))\di\Psi_t(x)\di t.
\end{equation}
On the other hand for any subinterval $[t_1,t_2]\subseteq [0,T]$, thanks to \cite[Remark 5.1.6]{AGS} we have that $\muo^N$ is narrowly convergent to $\muo$ in $\mathcal{M}([t_1,t_2]\times \R^d;\R^d)$. Therefore, applying Lemma \ref{flos1} for $X=[t_1,t_2]\times \R^d$ and Lemma \ref{flos2}, it holds 
\begin{eqnarray*}
\int_{t_1}^{t_2} \int_{\R^d} \phi(\wvo(t,x))\di\Psi_t(x)\di t \stackrel{\eqref{def-flos}} = \Phi(\muo|\Psi) \leq \liminf_{N\to+\infty} \Phi(\muo^N|\Psi^N) = \liminf_{N\to +\infty} \int_{t_1}^{t_2} \frac{1}{N}\sum_{i=1}^N\phi(\uo_i(t))\di t.
\end{eqnarray*}
Combining this with \eqref{gamma-conv2}, by standard argument in measure theory, we get 
\begin{eqnarray}
\label{gamma-conv3}
\lim_{N\to +\infty} \int_{t_1}^{t_2} \frac{1}{N}\sum_{i=1}^N\phi(\uo_i(t))\di t = \int_{t_1}^{t_2}\int_{\R^d} \phi(\wvo(t,x))\di\Psi_t(x)\di t.
\end{eqnarray}
\end{remark}

Our aim is to derive a first order optimality condition in the Wasserstein space $\mathcal{P}_1(\R^{2d})$ for the optimal control problem \eqref{costo}-\eqref{syst} which is limit as $N\to +\infty$ of the necessary condition given by the Pontryagin maximum principle for the finite optimal control problem \eqref{costoN}-\eqref{systN}. With this goal in mind, we introduce the rescaled costate variables (or adjoint variables) $\rv^N=(r_1,\dots,r_N)\in (\Rd)^N$. Thanks to the assumptions \Hv-\Hh-\HL, and using Lemma \ref{grad-identi}, we can apply the classical Pontryagin maximum principle to the optimal control problem (see also \cite[Theorem 3.5]{ADS}) obtaining the following result.
\begin{proposition}
\label{vecchio2}
Assume \HI, \Hv, \Hh, \HL \, and \Hphi. Let $(\xo^N,\uo^N)\in AC([0,T];(\Rd)^N) \times \K^N$ be an optimal trajectory-control pair for \eqref{costoN}-\eqref{systN}. Then there exists a costate curve $\ro^N\in AC([0,T];(\Rd)^N)$ such that $(\xo^N,\ro^N,\uo^N)$ is a solution to the system
\begin{equation}
\label{sys-tot-N}
\begin{cases}
\dis\dd \left(\begin{array}{c} \xo_i(t) \\ \ro_i(t) \end{array}\right) =
\left(\begin{array}{c}
N\nabla_{r_i} \mathcal H_N(\xo^N(t),\ro^N(t),\uo^N(t)) \\
-N\nabla_{x_i} \mathcal H_N(\xo^N(t),\ro^N(t),\uo^N(t))
\end{array}\right) & \text{ in } [0,T), \\
\xo_i (0)=\xo_{0,i}^N, \\
\ro_i(T) = 0, \\
\dis \uo^N (t)\in \argmax_{\uv^N\in K^N} \left\{ \mathcal H_N(\xo^N(t),\ro^N(t),\uv^N)\right\} & \text{ a.e. }t\in [0,T],
\end{cases}
\quad \text{for every }i=1,\dots,N,
\end{equation}
where the Hamiltonian $\mathcal H_N\colon  (\R^d)^N\times (\R^d)^N \times K^N \to \R$ is defined by
\begin{equation}
\label{ham-N}
\mathcal H_N(\xv^N,\rv^N,\uv^N):= \frac{1}{N} \sum_{k=1}^N \langle r_k, v(x_k,\psi^N) + h(x_k,\psi^N)u_k \rangle - L(\psi^N) - \frac{1}{N}\sum_{k=1}^N \phi(u_k).
\end{equation}
\end{proposition}
In view of the limit as $N\to +\infty$, it is useful to explicitly write the velocity field of system \eqref{sys-tot-N}. We have 
\begin{eqnarray}
\label{sys-field-N}
&&\left(\begin{array}{c}
\displaystyle N\nabla_{r_i} \mathcal H_N(\xo^N(t),\ro^N(t),\uo^N(t)) \\
\displaystyle -N\nabla_{x_i} \mathcal H_N(\xo^N(t),\ro^N(t),\uo^N(t))
\end{array}\right) \nonumber \\
&&=\left(\begin{array}{c}
\displaystyle  v(\xo_i(t),\Psi_t^N) + h(\xo_i(t),\Psi_t^N)\uo_i(t) \\
\displaystyle -\nabla_x^T v(\xo_i(t),\Psi_t^N)[\ro_i(t)] - \sum_{k=1}^N \nabla_{x_i}^T v(\xo_k(t),\Psi_t^N)[\ro_k(t)] + N \nabla_{x_i} L(\Psi_t^N)
\end{array}\right) \nonumber \\
&&+ \left(\begin{array}{c}
\displaystyle 0 \\
\displaystyle - \left(\nabla_x h(\xo_i(t),\Psi_t^N)\otimes \uo_i(t)\right)[\ro_i(t)] - \left(\sum_{k=1}^N \nabla_{x_i} h(\xo_k(t),\Psi_t^N)\otimes \uo_k(t)\right)[\ro_k(t)]
\end{array}\right) \nonumber \\
&&=\left(\begin{array}{c}
\displaystyle  v(\xo_i(t),\Psi_t^N) + h(\xo_i(t),\Psi_t^N)\uo_i(t) \\
\displaystyle -\nabla_x^T v(\xo_i(t),\Psi_t^N)[\ro_i(t)] - \sum_{k=1}^N \nabla_{x_i}^T v(\xo_k(t),\Psi_t^N)[\ro_k(t)] + N \nabla_{x_i} L(\Psi_t^N)
\end{array}\right) \nonumber \\
&&+ \left(\begin{array}{c}
\displaystyle 0 \\
\displaystyle - \nabla_x h(\xo_i(t),\Psi_t^N) \langle \ro_i(t),\uo_i(t)\rangle - \sum_{k=1}^N \nabla_{x_i} h(\xo_k(t),\Psi_t^N)\langle \ro_k(t), \uo_k(t)\rangle
\end{array}\right).
\end{eqnarray}
Considering the optimal state-costate-control $(\xo^N,\ro^N,\uo^N)\in AC([0,T];(\R^{2d})^N)\times \K^N$, we introduce the generated pairs $(\nuo^N,\rhoo^N)$ defined as
\begin{equation}
\label{def-nuo}
\nuo_t^N:= \frac{1}{N}\sum_{i=1}^N \delta_{(\xo_i(t),\ro_i(t))} \in \mathcal{P}_1(\R^{2d}), \qquad \nuo^N:=\nuo_t^N \otimes \mathcal{L}_{|[0,T]}\in C([0,T];\mathcal{P}_1(\R^{2d})),
\end{equation}
and
\begin{equation}
\label{def-rhoo}
\rhoo_t^N:= \frac{1}{N}\sum_{i=1}^N \uo_i(t)\delta_{(\xo_i(t),\ro_i(t))}\in \mathcal{M}(\R^{2d}; \R^d), \qquad \rhoo^N:=\rhoo_t^N \otimes \mathcal{L}_{|[0,T]}\in \mathcal{M}([0,T]\times \R^{2d}; \R^d).
\end{equation}
It follows from Lemma \ref{grad-identi} and \eqref{sys-field-N} that we can rewrite the equation in system \eqref{sys-tot-N} as
\begin{eqnarray}
\label{eq-N}
&& \dis\dd \left(\begin{array}{c} \xo_i(t) \\ \ro_i(t) \end{array}\right) \nonumber \\
&&=\left(\begin{array}{c}
\displaystyle  v(\xo_i(t),\Psi_t^N) + h(\xo_i(t),\Psi_t^N)\uo_i(t) \\
\displaystyle -\nabla_x^T v(\xo_i(t),\Psi_t^N)[\ro_i(t)] - \int_{\R^{2d}} \nabla^{T}_{\psi}v(\tilde x,\Psi_t^N)(\xo_i(t))[\tilde r] \di\nuo_t^N(\tilde x,\tilde r) + \nabla_{\psi} L(\Psi_t^N)(\xo_i(t))
\end{array}\right) \nonumber \\
&&+ \left(\begin{array}{c}
\displaystyle 0 \\
\displaystyle - \nabla_x h(\xo_i(t),\Psi_t^N) \langle \ro_i(t),\uo_i(t)\rangle - \int_{\R^{2d}} \nabla_{\psi} h(\tilde x,\Psi_t^N)(\xo_i(t))\langle \tilde r, \di\rhoo_t^N(\tilde x,\tilde r)\rangle
\end{array}\right).
\end{eqnarray}

We state our main result, i.e. the limit as $N\to +\infty$ of the necessary conditions for the finite optimal control problem \eqref{costoN}-\eqref{systN} contained in \eqref{sys-tot-N} which leads to necessary conditions for the Wasserstein optimal control problem \eqref{costo}-\eqref{syst}.
\begin{theorem}
\label{mainres}
Assume \HI, \Hv, \Hh, \HL \, and \Hphi. For every optimal trajectory-control pair $(\xo^N,\uo^N)\in AC([0,T];(\Rd)^N) \times \K^N$ for \eqref{costoN}-\eqref{systN}, let $(\Psi,\wvo)\in C([0,T];\PunoR)\times L^1_{\Psi}([0,T]\times \R^d;K)$ be the solution to the optimal control problem \eqref{costo}-\eqref{syst} given by Proposition \ref{vecchio}. Then there exists $\nuo\in \mathrm{Lip}([0,T]; \mathcal{P}_c(\R^{2d}))$ such that, up to subsequence, as $N\to +\infty$
\begin{itemize}
\item[$(a)$] $\nuo^N \to \nuo$ in $C([0,T]; \mathcal{P}_1(\R^{2d}))$ (where $\nuo^N$ is defined by \eqref{def-nuo});
\item[$(b)$] $\rhoo^N \to \wvo\nuo$ in the narrow topology of $\mathcal{M}([0,T]\times \R^{2d}; \R^d)$ (where $\rhoo^N$ is defined by \eqref{def-rhoo}).
\end{itemize}
Moreover $\nuo$ solves in distributional sense
\begin{equation}
\label{syst-tot-inf}
\begin{cases}
\dis\dd \nuo_t = - \mathrm{div}_{(x,r)}\left(\left(\begin{array}{c} \Gamma_1(x,\nuo_t,\wvo(t,x)) \\ \Gamma_2(x,r,\nuo_t,\wvo(t,x))\end{array}\right)\nuo_t\right) & \text{ in }
[0,T), \\
\pi^1_{\#}\nuo_t=\Psi_t & \text{ in } [0,T],\\
\nuo_T=\Psi_T \otimes \delta_{0} \in \mathcal{P}_c(\R^{2d}),
\end{cases}
\end{equation}
where
\begin{eqnarray*}
&& \left(\begin{array}{c} \Gamma_1(x,\nuo_t,\wvo(t,x)) \\ \Gamma_2(x,r,\nuo_t,\wvo(t,x))\end{array}\right) \\
&&= \left(\begin{array}{c} v(x,\pi_{\#}^1\nuo_t) + h(x,\pi_{\#}^1\nuo_t)\wvo(t,x) \\ -\nabla_x^{T}v(x,\pi_{\#}^1\nuo_t)[r] - \displaystyle\int_{\R^{2d}}\nabla_{\psi}^{T} v(\tilde x,\pi_{\#}^1\nuo_t)(x)[\tilde r] \di\nuo_t(\tilde x,\tilde r) + \nabla_{\psi}L(\pi_{\#}^1\nuo_t)(x)
\end{array}\right) \nonumber \\
&&+ \left(\begin{array}{c}
\displaystyle 0 \\
\displaystyle - \nabla_x h(x,\pi_{\#}^1\nuo_t) \langle r,\wvo(t,x)\rangle - \int_{\R^{2d}} \nabla_{\psi} h(\tilde x,\pi_{\#}^1\nuo_t)(x)\langle \tilde r,\wvo(t,\tilde x)\rangle \di\nuo_t(\tilde x,\tilde r)
\end{array}\right).
\end{eqnarray*}
Finally the following maximality condition holds
\begin{equation}
\label{cond-inf}
\displaystyle \wvo(t,\cdot)\in\argmax_{\omega\in L^1_{\Psi_t}(\R^d;K)} \left\{\mathcal{H}(\nuo_t,\omega(x)) \right\} \qquad \text{for a.e. }t\in [0,T],
\end{equation}
where $\mathcal{H}: \mathcal{P}_b(\R^{2d}) \times \mathcal{M}(\R^d;\R^d) \to \R$ is defined by
\begin{align}
\label{ham-inf}
\mathcal{H}(\nu,\omega):= \int_{\R^{2d}} & \langle r, v(x,\pi_{\#}^1\nu) + h(x,\pi_{\#}^1\nu)\omega(x)\rangle \di \nu(x,r) - L(\pi_{\#}^1\nu) 
\\
&
\nonumber- \int_{\R^{2d}} \phi(\omega(x)) \di \nu(x,r) 
\end{align}
if $\omega \in L^1_{\pi_{\#}^1\nu}(\R^d;K)$, and $\mathcal{H}(\nu,\omega) = +\infty$ otherwise.
\end{theorem}

\subsection{Comparison with the Pontryagin Maximum Principle for regular controls}
\label{s:compar}
The existing literature on first order optimality conditions for mean field optimal control problems as \eqref{costo}-\eqref{syst} (i.e., with closed loop structure of admissible controls) relies on an infinity dimensional version of the classical Pontryagin maximum principle which requires $C^1$-differentiability with respect to the space variable of the optimal control. We underline that the closed loop case is the most meaningful in the framework of mean-field optimal control as one can deduce from \cite[Section 6]{FLOS}. In this case, we have the following result. 
\begin{theorem}[{\cite[Theorem 5]{BonRos}} or {\cite[Theorem 4.10]{ADS}}]
\label{th:PontryRegular}
Under the assumptions of Theorem \ref{mainres}, assume, in addition, that $\phi\in C^1(\R^d;[0,+\infty))$. Let $(\Psi,\wvo)\in C([0,T];\PunoR)\times L^1_{\Psi}([0,T];U)$ be a solution to the optimal control problem \eqref{costo}-\eqref{syst} with $U$ a compact non-empty subset of $C^1_b(\R^d;K)$. Then there exists $\sigo\in AC([0,T];\mathcal{P}_c(\R^{2d}))$ which solves in distributional sense
\begin{equation*}
\begin{cases}
\dis\dd \sigo_t = - \mathrm{div}_{(x,r)}\left(\left(J\left(\nabla_{\psi}\mathcal{H}(\sigo_t,\wvo(t,x))(x,r)\right)\right)\sigo_t\right)    & \text{ in } [0,T), \\
\pi^1_{\#}\sigo_t=\Psi_t   & \text{ in } [0,T), \\
\sigo_T= \Psi_T \otimes \delta_{0} \in \mathcal{P}_c(\R^{2d}),
\end{cases}
\end{equation*}
where 
\begin{align*}
&  \left(\begin{array}{c} \Gamma_1(x,\sigo_t,\wvo(t,x)) \\ \Gamma_2(x,r,\sigo_t,\wvo(t,x))\end{array}\right)  =  J\left(\nabla_{\psi}\mathcal{H}(\sigo_t,\wvo(t,x))(x,r)\right)  + \beta (x, r, \sigo_{t}, \wvo)   
\end{align*}
with  $\R^d \ni (y_1,\dots,y_d)= y \mapsto \phi(y)\in \R$ and
\begin{align}
\beta(x, r, \sigo_{t}, \wvo(t, x) ):= \left(\begin{array}{c} 0 \\  \nabla_x^T\wvo(t,x)[h(x,\pi_{\#}^1\sigo_t)r - \nabla_y \phi(\wvo(t,x))]  \end{array}\right). \nonumber
\end{align} 
 \\
Moreover the following maximality condition holds
\begin{equation*}
\displaystyle \wvo(t,\cdot)\in\argmax_{\omega\in U} \left\{\mathcal{H}(\sigo_t,\omega(x)) \right\} \qquad \text{for a.e. }t\in [0,T].
\end{equation*}
\end{theorem}

A fundamental remark is that the assumption $\wvo(t,\cdot) \in C^1_b(\R^d;K)$ for almost every $t\in [0,T]$ it is very hard to satisfy for mean field optimal control problem as it is well argued in \cite{LipReg}. It follows that Theorem \ref{th:PontryRegular} is not always applicable. To better compare Theorems \ref{mainres} and \ref{th:PontryRegular}, we briefly discuss the following one-dimensional model case contained in \cite[Section 6]{LipReg}:
\begin{equation}
\label{comp1}
\displaystyle\min_{\omega\in L^1_{\psi}([0,T]\times \R; [-M,M])}\left\{\frac{\lambda}{2}\int_0^T \int_{\R^d} |\omega(t,x)|^2 \di \psi_t(x) \di t -\frac{1}{2}\int_{\mathbb{R}}|x-\overline{\psi_T}|^2 \di \psi_T(x) \right\} 
\end{equation}
subject to 
\begin{equation}
\label{comp2}
\begin{cases} \dd \psi_t = -\mathrm{div}_{x}\left(\omega(t,x)\psi_t\right) & \text{ in }(0,T],  \\
\psi_0= \frac{1}{2}\chi_{[-1,1]}\mathcal{L},
\end{cases}
\end{equation}
where $\displaystyle \overline{\psi_T}:=\int_{\mathbb{R}}x\,\di \psi_T(x)$ and $\lambda, M$ are two positive constants with $\lambda \leq T$. In the optimal control problem \eqref{comp1}-\eqref{comp2} one aims at maximizing the variance at time $T>0$  of a measure $\psi$, while penalizing the running $L^2_\psi$-norm of the control. \\
We fix a sequence of symmetrically distributed empirical measures $ \Psi^N_0:=\frac{1}{N}\sum_{i=1}^N \delta_{\xo_{0,i}^N}\in \mathcal{P}^N([-1,1])$ converging narrowly towards $\frac{1}{2}\chi_{[-1,1]}\mathcal{L}$. It follows that the finite particle optimal problem associated to \eqref{comp1}-\eqref{comp2} is 
\begin{equation}
\label{comp1N}
\displaystyle\min_{\uv^N\in L^1([0,T]; [-M,M]^N)}\left\{\frac{\lambda}{2N}\int_0^T \sum_{i=1}^N |u_i(t)|^2 \di t -\frac{1}{2N} \sum_{i=1}^N |x_i(T)-\overline{\xv(T)}|^2 \right\} 
\end{equation}
subject to 
\begin{equation}
\label{comp2N}
\begin{cases} \dd x_i(t) = u_i(t) & \text{ in }(0,T],  \\
x_i(0)= \xo^N_{0,i},
\end{cases}
\end{equation}
with $\displaystyle \overline{\xv(T)} :=\frac{1}{N}\sum_{i=1}^N x_i(t)$. Let $(\xo^N,\uo^N)$ be an optimal trajectory-control pair to \eqref{comp1}-\eqref{comp2} with generated pairs $(\Psi^N,\muo^N)$. Then, by Proposition \ref{vecchio} (which applies also in presence of a continuous final cost, since the proof depend only on the convexity of $\phi$), we have that $\Psi^N$ converges to $\Psi$ in $C([0,T];\mathcal{P}_1(\R))$ and $\muo^N$ converges weakly* to $\wvo(t,x)\Psi$ in $\mathcal{M} ([0, T]\times \R)$. Moreover $(\Psi,\wvo)$ is an optimal trajectory-control to problem \eqref{comp1}-\eqref{comp2}. In \cite[Proposition 9]{LipReg} it is proved that, since by assumption $\lambda\leq T$, then a uniform Lipschitz constant for the sequence $\uo^N$ of finite-dimensional optimal controls does not exist. Formally, this implies that the limit control $\wvo$ is not smooth making Theorem \ref{th:PontryRegular} inapplicable. On the other hand, we note that the final cost $\displaystyle \varphi(\mu):= -\frac{1}{2}\int_\R |x-\overline{\mu}|^2\di\mu(x)$ is continuously differentiable in the sense of Definition \ref{Wmudiff} and, by explicit calculation, we have
$$
\nabla_\psi\varphi(\mu) = -\mathrm{Id}_\R + \overline{\mu}.
$$
Thus we can apply Theorem \ref{mainres} (with minor modifications for $\nuo_T$ due to the final cost) obtaining that there exists $\nuo\in \mathrm{Lip}([0,T];\mathcal{P}_c(\R^{2}))$ which is a solution to the system
\begin{equation}
\label{comp3}
\begin{cases}
\dis\dd \nuo_t = - \mathrm{div}_{(x,r)}\left(\left(\begin{array}{c} \wvo(t,x) \\ 0 \end{array}\right)\nuo_t\right) & \text{ in }
[0,T), \\
\pi^1_{\#}\nuo_t=\Psi_t & \text{ in } [0,T],\\
\nuo_T=(\mathrm{Id}_\R, \mathrm{Id}_\R-\overline{\Psi_T})_{\#} \Psi_T, \\
\displaystyle \wvo(t,\cdot)\in\argmax_{\omega\in L^1_{\Psi_t}(\R^d;[-M,M])} \left\{\mathcal{H}(\nuo_t,\omega(x)) \right\} & \text{for a.e. }t\in [0,T],
\end{cases}
\end{equation}
where $\mathcal{H}: \mathcal{P}_b(\R^{2}) \times \mathcal{M}(\R;\R) \to \R$ is defined by
\begin{eqnarray*}
\mathcal{H}(\nu,\omega):= 
\begin{cases} \displaystyle
\int_{\R^{2}} r\omega(x) \di \nu(x,r) - \frac{\lambda}{2}\int_{\R^{2}} |\omega(x)|^2 \di \nu(x,r) &
\text{if }\omega \in L^1_{\pi_{\#}^1\nu}(\R^d;[-M,M]), \\
+\infty & \text{ otherwise}.
\end{cases} \nonumber
\end{eqnarray*}

We conclude by observing as in this simple model case it is clear how system \eqref{comp3} is the limit as $N\to +\infty$ of the classical first order optimality condition for the finite particle problem \eqref{comp1N}-\eqref{comp2N}, i.e. $\nuo$ is the limit of $\displaystyle \nuo^N= \frac{1}{N}\sum_{i=1}^N \delta_{(\xo_i(t),\ro_i(t))}\otimes \mathcal{L}_{|[0,T]}$ (see \eqref{def-nuo}) in $C([0,T];\mathcal{P}_1(\R^2))$ as $N\to+\infty$, where $(\xo^N,\ro^N,\uo^N)\in AC([0,T];(\R^2)^N)\times L^1([0,T]; [-M,M]^N)$ is the solution given by the classical Pontryagin maximum principle to 
\begin{equation*}
\begin{cases}
\dis\dd \left(\begin{array}{c} \xo_i(t) \\ \ro_i(t) \end{array}\right) =
\left(\begin{array}{c}
\uo_i(t) \\
0
\end{array}\right) & \text{ in } [0,T), \\
\xo_i (0)=\xo_{0,i}^N, \\
\ro_i(T) = \xo_i(T)-\overline{\xo(T)}, \\
\dis \uo^N (t)\in \argmax_{\uv^N\in [-M,M]^N} \left\{ \mathcal H_N(\xo^N(t),\ro^N(t),\uv^N)\right\} & \text{ a.e. }t\in [0,T],
\end{cases}
\quad \text{for every }i=1,\dots,N,
\end{equation*}
where the Hamiltonian $\mathcal H_N\colon  (\R^2)^N \times [-M,M]^N \to \R$ is defined by
\begin{equation*}
\mathcal H_N(\xv^N,\rv^N,\uv^N):= \frac{1}{N} \sum_{k=1}^N r_k u_k - \frac{\lambda}{2N}\sum_{k=1}^N |u_k|^2.
\end{equation*}

In Section \ref{s:convexcase} we will see how our results generalize to the case where the state space is the convex metric space $\R^d\times \mathcal{P}(U)$ equipped with the topology induced from the separable Banach space $(\R^d \times \mathcal{F}(U), |\cdot|+\|\cdot\|_{BL})$, for $U$ a discrete and finite set and if the controls act only on the $\R^d$ component of the velocity field. \\

\section{Proof of the results}
\label{s:proofs}

To lighten the notation, in all this section, we will denote with $\bm{(\mathrm H)}$ all the assumptions \HI, \Hv, \Hh, \HL, \Hphi. Moreover, any time we write $\|\cdot\|$ we mean the norm of the space of matrices $\R^{d \times d}$. \\
We recall that $\hat{\Psi}_0\in \PCR$ is a fixed initial measure and $\xo_0^N\in (\R^d)^N$ is a fixed sequence of initial data satisfying \HI. Moreover, for every optimal trajectory-control pair $(\xo^N,\uo^N)\in AC([0,T];(\Rd)^N) \times \K^N$ for \eqref{costoN}-\eqref{systN}, $(\Psi,\wvo)\in C([0,T];\PunoR)\times L^1_{\Psi}([0,T]\times \R^d;K)$ is the solution to the optimal control problem \eqref{costo}-\eqref{syst} given by Proposition \ref{vecchio} and $\ro^N\in AC([0,T];(\R^{d})^N)$ the rescaled costate given by Proposition \ref{vecchio2}. \\

With the aim of proving Theorem \ref{mainres} we give some preliminary lemmas. In the first one we prove that the solution $(\xo_i(t), \ro_i(t))$ of system \eqref{sys-tot-N} is contained in a compact set uniformly in $N\in\N$. 

\begin{lemma}
\label{lemma1}
Assume $\bm{(\mathrm H)}$. Let $(\xo^N,\ro^N,\uo^N)\in AC([0,T];(\R^{2d})^N) \times \K^N$ be the solution to \eqref{sys-tot-N} given by Proposition \ref{vecchio2}. Then there exists $\mathcal{R}>0$ depending only on $\supp(\hat{\Psi}_0)$ and $T$ and independent of~$N$ such that
\begin{equation}
\label{bnd-N}
\sup_{i=1,\dots,N}\|\xo_i\|_{L^{\infty}([0,T];\R^d)} + \sup_{i=1,\dots,N}\|\ro_i\|_{L^{\infty}([0,T];\R^d)} \leq \mathcal{R}.
\end{equation}
Equivalently, it holds that 
\begin{equation}
\label{bnd-nu}
\supp(\nuo_t^N)\subseteq B_{\mathcal{R}}(0)\subset \R^{2d} \qquad \forall t\in [0,T] \text{ and } \forall N\in \N.
\end{equation}
\end{lemma}

\begin{proof}
First we note that, thanks to \cite[Proposition 1]{AAMS} and recalling that $\xo_{0,i}^N\in B_1^{\hat{\Psi}_0}$ for every $i=1,\dots,N$ and for all $N\in\N$, we have
\begin{equation}
\label{pr1}
\sup_{i=1,\dots,N}\|\xo_i\|_{L^{\infty}([0,T];\R^d)} \leq R_1,
\end{equation}
for some $R_1>0$ dependent on $\supp(\hat{\Psi}_0)$ and $T$ and independent of $N$. Therefore, by \eqref{def-psi}, we deduce that $\supp(\Psi^N_t)\subset B_{R_1}(0) \subset \R^d$ for every $t\in [0,T]$ and $N\in \N$. This implies that $\Psi^N_t$ is a tight and $1$-uniformly integrable sequence in $\mathcal{P}_1(\R^d)$. Thus, by \cite[Proposition 7.1.5]{AGS}, there exists a compact subset $K_1$ of $\mathcal{P}_1(\R^d)$ such that $\Psi^N_t\in K_1$ for every $t\in [0,T]$ and for any $N\in\N$. Moreover, by \cite[Proposition 5.1.8]{AGS}, up to a subsequence in $N$, $K_1\subset \mathcal{P}_c(\R^d)$. \\
Now we focus on $\ro_i$. By Lemma \ref{grad-identi}, we have the following identifications:
$$\nabla_{x_i}^T v(\xo_k(t),\Psi_t^N) = \frac{1}{N}\nabla^{T}_{\psi}v(\xo_k(t),\Psi_t^N)(\xo_i(t)), \qquad N\nabla_{x_i} L(\Psi_t^N) = \nabla_{\psi} L(\Psi_t^N)(\xo_i(t))
$$ and 
$$
\nabla_{x_i} h(\xo_k(t),\Psi_t^N) = \frac{1}{N}\nabla_{\psi} h(\xo_k(t),\Psi_t^N)(\xo_i(t)).
$$
Hence, it follows from \eqref{sys-tot-N} and \eqref{sys-field-N} that
\begin{eqnarray}
\label{pr2}
|\ro_i(t)|& \leq & \int_t^T \|\nabla_x^T v(\xo_i(\tau),\Psi_\tau^N)\||\ro_i(\tau)|\di \tau + \int_t^T \frac{1}{N}\sum_{k=1}^N  \|\nabla^{T}_{\psi}v(\xo_k(\tau),\Psi_\tau^N)(\xo_i(\tau))\||\ro_k(\tau)| \di \tau \nonumber \\
&& + \int_t^T | \nabla_{\psi} L(\Psi_\tau^N)(\xo_i(\tau))| \di \tau \nonumber + \int_t^T |\nabla_x h(\xo_i(\tau),\Psi_\tau^N)| |\ro_i(\tau)||\uo_i(\tau)| \di \tau \nonumber \\
&& + \int_t^T \frac{1}{N}\sum_{k=1}^N |\nabla_{\psi} h(\xo_k(\tau),\Psi_\tau^N)(\xo_i(\tau))| |\ro_k(\tau)||\uo_k(\tau)|\di \tau.
\end{eqnarray}
Now, since $\uo_i(\tau)\in K$ which is compact in $\R^d$, $\Psi^N_\tau\in K_1$ which is compact in $\mathcal{P}_c(\R^d)$ and $\xo_i(\tau)\in B_{R_1}(0)$ which is compact in $\R^d$ for every $i=1,\dots,N$, for any $N\in \N$ and for every $\tau\in [0,T]$, using the continuity assumptions \Hv-$(v_3,v_4)$, \Hh-$(h_3,h_4)$ and \HL-$(L_2)$, we have for some positive constant $M$ not depending on $i$, $k$, $N$ and $\tau$ that
\begin{eqnarray}
\label{pr3}
&&\|\nabla_x^T v(\xo_i(\tau),\Psi_\tau^N)\| + \|\nabla^{T}_{\psi}v(\xo_k(\tau),\Psi_\tau^N)(\xo_i(\tau))\|  + | \nabla_{\psi} L(\Psi_\tau^N)(\xo_i(\tau))| \nonumber \\
&&+ |\nabla_x h(\xo_i(\tau),\Psi_\tau^N)||\uo_i(\tau)| + |\nabla_{\psi} h(\xo_k(\tau),\Psi_\tau^N)(\xo_i(\tau))||\uo_k(\tau)| \leq M.
\end{eqnarray}
Combining \eqref{pr2} and \eqref{pr3} we obtain
\begin{eqnarray*}
|\ro_i(t)|& \leq & M(T-t) + 2M\left(\int_t^T |\ro_i(\tau)|\di\tau + \int_t^T \frac{1}{N}\sum_{k=1}^N |\ro_k(\tau)|\di \tau\right) \nonumber \\
&\leq& M(T-t) + 4M \int_t^T \sup_{i=1,\dots,N}|\ro_i(\tau)| \di \tau.
\end{eqnarray*}
Taking the supremum over $i\in\{1,\dots,N\}$ in the previous inequality and applying Gr\"onwall inequality we deduce for some positive $R_2>0$ depending on $M$ and $T$ and not depending on $t\in[0,T]$ and $N\in \N$ that
\begin{equation}
\label{pr4}
\sup_{i=1,\dots,N}|\ro_i(t)| \leq R_2 \qquad\forall t\in [0,T].
\end{equation}
Thus \eqref{bnd-N} follows immediately from \eqref{pr1} and \eqref{pr4}. Finally, \eqref{bnd-nu} is a direct consequence of \eqref{bnd-N} and of the definition of $\nuo_t^N$ (i.e. \eqref{def-nuo}).
\end{proof}

In the second one we show that $\nuo^N$ defined in \eqref{def-nuo} is Lipschitz continuous in time with uniform Lipschitz constant in $N\in \N$.

\begin{lemma}
\label{lemma2}
Assume $\bm{(\mathrm H)}$. Then there exists $L$ not depending on $N\in \N$ such that 
\begin{equation}
\label{lip-nu}
W_1(\nuo_t^N,\nuo_s^N) \leq L|t-s| \qquad \forall s,t\in [0,T].
\end{equation}
\end{lemma}

\begin{proof}
By definition of $\nuo_t^N$ (see \eqref{def-nuo}) and by the properties of the Wasserstein distance, it holds that
\begin{equation}
\label{pr5}
W_1(\nuo_t^N,\nuo_s^N) \leq \frac{\sqrt{2}}{N}\sum_{i=1}^N \left(|\xo_i(t)-\xo_i(s)| + |\ro_i(t)-\ro_i(s)|\right).
\end{equation}
First we prove that $L_1>0$ not depending on $i$ and $N$ exists such that  
\begin{equation}
\label{pr6}
|\xo_i(t)-\xo_i(s)| \leq L_1 |t-s|.
\end{equation}
By \eqref{sys-tot-N} and \eqref{sys-field-N} we have
\begin{eqnarray}
\label{pr7}
|\xo_i(t)-\xo_i(s)| \leq \int_s^t \left(|v(\xo_i(\tau),\Psi_\tau^N)| + |h(\xo_i(\tau),\Psi_\tau^N)||\uo_i(\tau)|\right) \di\tau.
\end{eqnarray}
From the proof of Lemma \ref{lemma1}, we know that $\uo_i(\tau)\in K$ which is compact in $\R^d$, $\Psi^N_\tau\in K_1$ which is compact in $\mathcal{P}_c(\R^d)$ and $\xo_i(\tau)\in B_{R_1}(0)$ which is compact in $\R^d$ for every $i=1,\dots,N$, for any $N\in \N$ and for every $\tau\in [0,T]$. Since, by assumptions, $v$ and $h$ are continuous, then $L_1>0$ not depending on $i$, $N$ and $\tau$ exists such that
$$
|v(\xo_i(\tau),\Psi_\tau^N)| + |h(\xo_i(\tau),\Psi_\tau^N)||\uo_i(\tau)|\leq L_1.
$$
The previous inequality combined with \eqref{pr7} gives \eqref{pr6}. \\
Now we focus on the second term on the right-hand side of \eqref{pr5}. Using again \eqref{sys-tot-N} and \eqref{sys-field-N} and following the proof of Lemma \ref{lemma1} (in particular using \eqref{pr3}) we obtain
\begin{eqnarray}
\label{pr8}
|\ro_i(t)-\ro_i(s)|& \leq & \int_s^t \|\nabla_x^T v(\xo_i(\tau),\Psi_\tau^N)\||\ro_i(\tau)|\di \tau \nonumber \\
&&+ \int_s^t \frac{1}{N}\sum_{k=1}^N  \|\nabla^{T}_{\psi}v(\xo_k(\tau),\Psi_\tau^N)(\xo_i(\tau))\||\ro_k(\tau)| \di \tau \nonumber \\
&& + \int_s^t | \nabla_{\psi} L(\Psi_\tau^N)(\xo_i(\tau))| \di \tau \nonumber + \int_t^T |\nabla_x h(\xo_i(\tau),\Psi_\tau^N)| |\ro_i(\tau)||\uo_i(\tau)| \di \tau \nonumber \\
&& + \int_s^t \frac{1}{N}\sum_{k=1}^N |\nabla_{\psi} h(\xo_k(\tau),\Psi_\tau^N)(\xo_i(\tau))| |\ro_k(\tau)||\uo_k(\tau)|\di \tau \nonumber \\
&\stackrel{\eqref{bnd-N},\eqref{pr3}}\leq & M(1+4\mathcal{R})|t-s|.
\end{eqnarray}
Therefore, inserting \eqref{pr6} and \eqref{pr8} in \eqref{pr5}, we deduce \eqref{lip-nu}.
\end{proof}

Thanks to Lemma \ref{lemma1} and Lemma \ref{lemma2}, in the next result we prove that $\nuo^N$ (defined by \eqref{def-nuo}) and $\rhoo^N$ (defined by \eqref{def-rhoo}) admit limit and we characterize the limit of $\rhoo^N$ in terms of the limit of $\muo^N$ (defined by \eqref{def-muo}).

\begin{lemma}
\label{lemma3}
Assume $\bm{(\mathrm H)}$. Then the following hold (up to a subsequence):
\begin{itemize}
\item[$(a)$] there exists $\nuo\in \mathrm{Lip}([0,T];\mathcal{P}_c(\R^{2d}))$ such that $\nuo^N \to \nuo$ in $C([0,T];\mathcal{P}_1(\R^{2d}))$ as $N\to +\infty$;
\item[$(b)$] there exists $\rhoo \in \mathcal{M}([0,T]\times \R^{2d}; \R^d)$ such that $\rhoo^N \to \rhoo$ narrowly in $\mathcal{M}([0,T]\times \R^{2d}; \R^d)$ as $N\to +\infty$ with $\rhoo=\wvo(t,x)\nuo$, where $\wvo$ is given by Proposition \ref{vecchio}.
\end{itemize}
\end{lemma}

\begin{proof}
\textit{Proof of $(a)$.}
By Lemma \ref{lemma1}, we know that $\supp(\nuo^N_t)\subset B_{\mathcal{R}}(0) \subset \R^{2d}$ for every $t\in [0,T]$ and $N\in \N$. This implies that $\nuo^N_t$ is a tight and $1$-uniformly integrable sequence in $\mathcal{P}_1(\R^{2d})$. Then, by \cite[Proposition 7.1.5]{AGS}, $\nuo^N_t$ is relatively compact in $\mathcal{P}_1(\R^{2d})$ for every $t\in [0,T]$. Moreover, by Lemma \ref{lemma2}, $\nuo^N_t$ is equi-Lipschitz continuous in $t$. Thus we can apply the Ascoli-Arzel\'a theorem, obtaining that there exists $\nuo\in C([0,T];\mathcal{P}_1(\R^{2d}))$ such that, up to a subsequence, $\nuo^N \to \nuo$ in $C([0,T];\mathcal{P}_1(\R^{2d}))$ as $N\to +\infty$. The fact that $\nuo_t \in \mathcal{P}_c(\R^{2d})$ for every $t\in [0,T]$ follows from \cite[Proposition 5.1.8]{AGS} and the Lipschitz continuity of $\nuo$ follows from \eqref{lip-nu} and applying \cite[Proposition 7.1.3]{AGS}. \\
\textit{Proof of $(b)$}. Since $\uo_i(t)\in K$ (compact) for every $t\in [0,T]$, we have for some $M>0$ that 
$$
\|\rhoo^N\|_{\mathcal{M}([0,T]\times \R^{2d}; \R^d)} \stackrel{\eqref{def-rhoo}}\leq \int_0^T \frac{1}{N} \sum_{i=1}^N |\uo_i(t)| \leq MT.
$$
Moreover, by \eqref{def-rhoo} and \eqref{bnd-nu}, $\supp(\rhoo^N)\subseteq [0,T] \times B_\mathcal{R}(0)$, which implies that $\rhoo^N$ is tight in $\mathcal{M}([0,T]\times \R^{2d}; \R^d)$. Hence, applying Prokhorov theorem, there exists $\rhoo\in \mathcal{M}([0,T]\times \R^{2d}; \R^d)$ such that, up to a subsequence, $\rhoo^N$ narrowly converges to $\rhoo$ as $N\to +\infty$. Thanks to Lemma \ref{flos2} and Lemma \ref{flos1} there exists $\eta\in L^1_{\nuo}([0,T]\times \R^{2d};\R^d)$ such that $\rhoo=\eta(t,x,r)\nuo$. \\
Finally we prove that $\rhoo=\wvo\nuo$, i.e. $\eta(t,x,r)=\wvo(t,x)$ for every $(t,x,r)\in [0,T]\times \R^{2d}$. This is done in two steps. Let $\pi^{0,1}:[0,T]\times \R^{2d}\to [0,T]\times \R^d$ be the projection defined as $\pi^{0,1}(t,x,r)=(t,x)$. Then, as $\rhoo^N$ converges narrow to $\rhoo$ and $\pi_{\#}^{0,1}\rhoo^N = \muo^N$, by \cite[Lemma 5.2.1]{AGS}, it holds $\pi_{\#}^{0,1}\rhoo=\muo=\wvo\Psi$. Similarly we have $\pi_{\#}^{0,1}\nuo=\Psi$. Applying the Disintegration Theorem (see \cite[Theorem 5.3.1]{AGS}) with respect to $\Psi=\pi_{\#}^{0,1}\nuo$, we get $\rhoo=\eta(t,x,r)\nuo^{x}_t\otimes \Psi$ for $\Psi$-a.e. $(t,x)\in [0,T]\times \R^d$ where $\nuo_t^x\in \mathcal{P}(\R^d)$. It now holds
\begin{equation}
\label{pr9bis}
\int_{\R^d} \eta(t,x,r)\di \nuo_t^x(r) = \wvo(t,x),
\end{equation}
hence it is sufficient to show that $\eta$ does not depend on $r$. \\
We start noticing that, applying Lemma \ref{flos2} and Lemma \ref{flos1},
\begin{eqnarray}
\label{pr9}
&& \int_0^T \int_{\R^d} \phi(\wvo(t,x))\di\Psi_t(x)\di t = \int_0^T \int_{\R^d} \phi\left(\frac{\di\muo}{\di\Psi}\right) \di \Psi_t(x)\di t \nonumber \\
&&\stackrel{\eqref{gamma-conv2}}= \liminf_{N\to +\infty} \int_0^T \int_{\R^d} \phi\left(\frac{\di\muo^N}{\di \Psi^N}\right)\di \Psi_t^N(x)\di t  = \liminf_{N\to +\infty} \int_0^T \frac{1}{N}\sum_{i=1}^N \phi(\uo_i(t))\di t \nonumber \\
&& \geq \liminf_{N\to +\infty} \int_0^T \int_{\R^{2d}} \phi\left(\frac{\di \rhoo^N}{\di \nuo^N} \right)\di \nuo^N_t(x,r)\di t \geq \int_0^T \int_{\R^{2d}} \phi\left(\frac{\di \rhoo}{\di \nuo}\right) \di \nuo_t(x,r)\di t.
\end{eqnarray}
We now proceed in the above inequality, by using Jensen's inequality and the properties of the disintegration, as follows
\begin{eqnarray*}
&&\int_0^T \int_{\R^d} \phi(\wvo(t,x))\di\Psi_t(x)\di t  \stackrel{\eqref{pr9}}\geq  \int_0^T \int_{\R^{2d}} \phi\left(\eta(t,x,r)\right) \di \nuo_t(x,r)\di t \\
&& =\int_0^T \int_{\R^d}\int_{\R^d} \phi(\eta(t,x,r))\di \nuo_t^x(r)\di \Psi_t(x)\di t \geq \int_0^T \int_{\R^d} \phi\left(\int_{\R^d}\eta(t,x,r)\di \nuo_t^x(r)\right) \di \Psi_t(x)\di t \\
&& \stackrel{\eqref{pr9bis}}= \int_0^T \int_{\R^d} \phi(\wvo(t,x))\di\Psi_t(x)\di t.
\end{eqnarray*}
Therefore, for $\Psi$-a.e. $(t,x)\in [0,T] \times \R^d$, the equality case in Jensen's inequality must hold. As $\phi$ is strictly convex (see \Hphi), the only possibility is that $\eta$ does not depend on $r$ which in turn implies, by \eqref{pr9bis}, that 
$$
\eta(t,x) = \int_{\R^d} \eta(t,x,r) \di\nuo_t^x(r) = \wvo(t,x).
$$
\end{proof}

Now we are ready to pass the system \eqref{eq-N} to the limit as $N\to +\infty$ and to prove our main result.

\begin{proof}[Proof of Theorem \ref{mainres}]
We start by noting that $(a)$ and $(b)$ in the statement are given directly by Lemma \ref{lemma3}. \\
Now we focus on the second and the third equality of system \eqref{syst-tot-inf}. Let us recall for the following that if a sequence of measure converges in the Wasserstein space $\mathcal{P}_1(\R^{d})$ then it also converges narrowly. As regards the second equality of \eqref{syst-tot-inf}, for every Borel set $B\subseteq \R^d$ and for every $t\in[0,T]$ it holds 
\begin{eqnarray}
\label{pr10}
\pi^1_{\#}\nuo^N_t(B) &\stackrel{\eqref{def-nuo}}=&\nuo^N_t(B\times\R^d)=\frac{1}{N}\sum_{i=1}^N \delta_{(\xo_i(t),\ro_i(t))}(B\times \R^d) \nonumber \\
&=&\frac{1}{N}\sum_{i=1}^N \delta_{\xo_i(t)}(B) \stackrel{\eqref{def-psi}}=\Psi_t^N(B).
\end{eqnarray}
Moreover, by Lemma \ref{lemma3} and applying \cite[Lemma 5.2.1]{AGS}, $\pi^1_{\#}\nuo^N_t$ narrowly converges to $\pi^1_{\#}\nuo_t$ and, by Proposition \ref{vecchio}, $\Psi_t^N$ narrowly converges to $\Psi_t$ as $N\to +\infty$. These convergences together with \eqref{pr10} imply that 
$$
\pi^1_{\#}\nuo_t=\Psi_t \quad  \text{for every }t\in [0,T].
$$
As for the third equality of \eqref{syst-tot-inf}, we have
$$
\nuo^N_T \stackrel{\eqref{def-nuo}}= \frac{1}{N}\sum_{i=1}^N\delta_{(\xo_i(T),\ro_i(T))} \stackrel{\eqref{sys-tot-N}} = \frac{1}{N}\sum_{i=1}^N\delta_{(\xo_i(T),0)} = \frac{1}{N}\sum_{i=1}^N\delta_{\xo_i(T)} \otimes  \delta_0 \stackrel{\eqref{def-psi}} = \Psi^N_T \otimes \delta_0.
$$
Since, by Lemma \ref{lemma3}, $\nuo^N_T$ narrowly converges to $\nuo_T$ and, by Proposition \ref{vecchio}, $\Psi^N_T$ narrowly converges to $\Psi_T$ as $N\to +\infty$, it follows from the previous equality that 
$$
\nuo_T=\Psi_T \otimes \delta_0.
$$
Now we prove that $\nuo$ solves the continuity equation in \eqref{syst-tot-inf} in the sense of distributions. Thanks to \eqref{eq-N} and by \eqref{def-nuo}, for every test function $\varphi\in C^{\infty}_c((0,T)\times \R^{2d})$ we have that, for every $t\in [0,T]$, 
\begin{eqnarray}
\label{pr11}
&& \int_{\R^{2d}} \varphi(t,x,r) \di\nuo^N_t(x,r) = \int_0^t \int_{\R^{2d}}\partial_t\varphi(\tau,x,r) \di\nuo^N_\tau(x,r)\di \tau \nonumber \\
&& +  \int_0^t \int_{\R^{2d}} \langle \nabla_x\varphi(\tau,x,r), v(x,\Psi_\tau^N)\rangle \di\nuo^N_\tau(x,r)\di \tau  \nonumber \\
&& + \int_0^t \int_{\R^{2d}} \langle \nabla_x\varphi(\tau,x,r), h(x,\Psi_\tau^N)\di\rhoo^N_\tau(x,r)\rangle \di \tau \nonumber \\
&& - \int_0^t \int_{\R^{2d}} \langle \nabla_r\varphi(\tau,x,r), \nabla_x^T v(x,\Psi_\tau^N)[r]\rangle \di\nuo^N_\tau \di \tau \nonumber \\
&& - \int_0^t \int_{\R^{2d}}\int_{\R^{2d}} \langle \nabla_r\varphi(\tau,x,r), \nabla^{T}_{\psi}v(\tilde x,\Psi_\tau^N)(x)[\tilde r] \rangle \di\nuo_\tau^N(\tilde x,\tilde r) \di\nuo^N_\tau(x,r)\di \tau \nonumber \\
&& + \int_0^t \int_{\R^{2d}} \langle \nabla_r\varphi(\tau,x,r), \nabla_{\psi} L(\Psi_\tau^N)(x)\rangle \di\nuo^N_\tau(x,r)\di \tau \nonumber \\
&& - \int_0^t \int_{\R^{2d}} \langle \nabla_r\varphi(\tau,x,r),\nabla_x h(x,\Psi_\tau^N)\rangle \langle r,\di\rhoo^N_\tau(x,r)\rangle \di \tau \nonumber \\
&& - \int_0^t \int_{\R^{2d}}\int_{\R^{2d}} \langle \nabla_r\varphi(\tau,x,r), \nabla_{\psi} h(\tilde x,\Psi_\tau^N)(x)\rangle\langle \tilde r, \di\rhoo_\tau^N(\tilde x,\tilde r)\rangle \di\nuo^N_\tau(x,r)\di \tau.
\end{eqnarray}
Since, by Lemma \ref{lemma3}, $\nuo^N\to\nuo$ in $C([0,T];\mathcal{P}_1(\R^{2d}))$ and $\rhoo^N\to \rhoo=\wvo(t,x)\nuo$ narrowly in $\mathcal{M}([0,T]\times \R^{2d};\R^d)$ as $N\to+\infty$, the integral on the left-hand side and the first integral on the right-hand side of \eqref{pr11} immediately pass to the limit. For all other integrals on the right-hand side of \eqref{pr11} we follow the same technique to pass to the limit as $N\to +\infty$. In light of this fact, for brevity, we deal with only the last integral on the right-hand side (which contains a double integration). To do this, we define $\rhoo_\tau:=\wvo(\tau,\cdot)\nuo_\tau \in \mathcal{M}(\R^{2d};\R^d)$ and we estimate 
\begin{eqnarray}
\label{pr12}
&& \left|\int_0^t \int_{\R^{2d}}\int_{\R^{2d}} \langle \nabla_r\varphi(\tau,x,r), \nabla_{\psi} h(\tilde x,\Psi_\tau^N)(x)\rangle\langle \tilde r, \di\rhoo_\tau^N(\tilde x,\tilde r)\rangle \di\nuo^N_\tau(x,r)\di \tau \right.  \nonumber \\
&& \left. -\int_0^t \int_{\R^{2d}}\int_{\R^{2d}} \langle \nabla_r\varphi(\tau,x,r), \nabla_{\psi} h(\tilde x,\Psi_\tau)(x)\rangle\langle \tilde r, \di\rhoo_\tau(\tilde x,\tilde r)\rangle \di\nuo_\tau(x,r)\di \tau \right| \nonumber \\
&& \leq \left|\int_0^t\int_{\R^{4d}}\langle \nabla_r\varphi(\tau,x,r),\nabla_{\psi} h(\tilde x,\Psi_\tau^N)(x)-\nabla_{\psi} h(\tilde x,\Psi_\tau)(x)\rangle \langle \tilde r, \di\rhoo_\tau^N(\tilde x,\tilde r)\rangle \di\nuo^N_\tau(x,r)\di \tau\right | \nonumber \\
&& + \left|\int_0^t\int_{\R^{4d}} \langle \nabla_r\varphi(\tau,x,r), \nabla_{\psi} h(\tilde x,\Psi_\tau)(x)\rangle\langle \tilde r, \di\rhoo_\tau^N(\tilde x,\tilde r)\di\nuo^N_\tau(x,r) -\di\rhoo_\tau(\tilde x,\tilde r) \di\nuo_\tau(x,r)\rangle \di \tau \right| \nonumber \\
&& =: I_1^N + I_2^N.
\end{eqnarray}
First we focus on $I_1^N$. We have
$$
I_1^N \leq \int_0^t |G^N(\tau)| \di \tau,
$$
where
$$
G^N(\tau):= \int_{\R^{4d}}\langle \nabla_r\varphi(\tau,x,r),\nabla_{\psi} h(\tilde x,\Psi_\tau^N)(x)-\nabla_{\psi} h(\tilde x,\Psi_\tau)(x)\rangle \langle \tilde r, \di\rhoo_\tau^N(\tilde x,\tilde r)\rangle \di\nuo^N_\tau(x,r).
$$
We recall that, by Lemma \ref{lemma1}, $\supp(\nuo^N_\tau)\stackrel{\eqref{def-rhoo}}=\supp(\rhoo^N_\tau) \subseteq B_{\mathcal{R}}(0)\subset \R^{2d}$ and
$\Psi_\tau^N \subset K_1$ where $K_1$ is a compact subset of $\mathcal{P}_c(\R^d)$, for every $N\in \N$ and $\tau\in [0,T]$. Hence, $\nabla_\psi h(\tilde x, \Psi_\tau)(x)$ is continuous on a compact subset of $\R^d\times \mathcal{P}_c(\R^d)\times \R^d$ and there exists $\omega: [0,+\infty)\to [0,+\infty]$ modulus of continuity with $\lim_{s\to 0^+}\omega(s)=0$ such that
$$
|\nabla_{\psi} h(\tilde x,\Psi_\tau^N)(x)-\nabla_{\psi} h(\tilde x,\Psi_\tau)(x)| \leq \omega\left(W_1(\Psi_\tau^N,\Psi_\tau)\right).
$$
It follows from definition of $\rhoo^N$ and $\nuo^N$ (i.e. \eqref{def-rhoo}-\eqref{def-nuo}) and since $\uo\in [-M,M]^N$ for some $M>0$, that 
$$
|G^N(\tau)|\leq \|\nabla_r\varphi\|_{L^\infty([0,T]\times \R^{2d})} \mathcal{R} M \omega\left(W_1(\Psi_\tau^N,\Psi_\tau)\right),
$$
which, using that, by Proposition \ref{vecchio}, $\Psi_\tau^N \to \Psi_\tau$ in $\mathcal{P}_1(\R^d)$, implies
$$
\lim_{N\to+\infty} G^N(\tau)=0 \qquad \text{a.e. }\tau \in [0,T].
$$
This fact, noting that $G^N(\tau)$ is uniformly bounded in $[0,T]$ and applying the Lebesgue theorem, leads to
\begin{equation}
\label{pr13}   
0\leq\lim_{N\to +\infty} I_1^N \leq \lim_{N\to +\infty} \int_0^t|G^N(\tau)| \di\tau = 0.
\end{equation}
We are left to prove that $I_2^N \to 0$ as $N\to +\infty$. We notice that, since $\Psi\in C([0,T];\mathcal{P}_1(\R^d))$ and by $\bm{(\mathrm H)}$, we have
\begin{equation*}
\langle \nabla_r\varphi(\tau,x,r), \nabla_{\psi} h(\tilde x,\Psi_\tau)(x)\rangle \tilde r \in C([0,T]\times B_{\mathcal{R}}(0)\times B_{\mathcal{R}}(0);\R^d).
\end{equation*}
Hence, using the density of the linear span of test functions of the form $\varphi(\tau,x,r,\tilde x,\tilde r)=\alpha(\tau)\theta(x,r)\beta(\tilde x,\tilde r)$ in $C([0,T]\times B_{\mathcal{R}}(0)\times B_{\mathcal{R}}(0))$ with $\alpha \in C( [0, T])$, $\theta \in C ( B_{\mathcal{R}} (0))$ and $\beta \in {\rm Lip} ( B_{\mathcal{R}} (0))$, it is enough to show that for every $\alpha\in C([0,T])$, $\theta \in C ( B_{\mathcal{R}} (0))$ and $\beta \in {\rm Lip} ( B_{\mathcal{R}} (0))$ it holds
\begin{align}
\label{pr14bis}
\lim_{N \to \infty} \int_0^t\int_{\R^{4d}} & \alpha( \tau) \beta(\tilde{x}, \tilde{r}) \theta(x, r)\di\rhoo_\tau^N(\tilde x,\tilde r)\di\nuo^N_\tau(x,r)\di \tau \nonumber \\
& =\int_0^t\int_{\R^{4d}} \alpha( \tau) \beta(\tilde{x}, \tilde{r}) \theta(x, r) \di\rhoo_\tau(\tilde x,\tilde r) \di\nuo_\tau(x,r)  \di \tau \,. 
\end{align}
By simple algebraic manipulations, we write
\begin{align}
\label{pr14tris}
\int_0^t\int_{\R^{4d}} & \alpha( \tau) \beta(\tilde{x}, \tilde{r}) \theta(x, r)\di\rhoo_\tau^N(\tilde x,\tilde r)\di\nuo^N_\tau(x,r)\di \tau 
\\
&
= \int_0^t \alpha( \tau) \int_{\R^{2d}} \bigg( \int_{\R^{2d}}  \theta(x, r) \di (\nuo_{\tau}^{N}  - \nuo_{\tau}) (x, r) \bigg) \beta(\tilde{x}, \tilde{r})\di\rhoo_\tau^N(\tilde x,\tilde r) \di \tau \nonumber
\\
&
\qquad + \int_0^t \alpha( \tau) \int_{\R^{2d}} \bigg( \int_{\R^{2d}}  \theta(x, r) \di \nuo_{\tau} (x, r) \bigg) \beta(\tilde{x}, \tilde{r})\di\rhoo_\tau^N(\tilde x,\tilde r) \di \tau \nonumber\,.
\end{align}
For the first term on the right-hand side of~\eqref{pr14tris}, by uniform convergence of $\nu^{N}_{\tau}$ to $\nu_{\tau}$ in the 1-Wasserstein distance and recalling that $K$ is compact in $\R^d$, we get for some $M>0$ that
\begin{align}
\label{pr14quarter}
\lim_{N\to \infty} \int_0^t & \alpha( \tau) \int_{\R^{2d}} \bigg( \int_{\R^{2d}}  \theta(x, r) \di (\nuo_{\tau}^{N}  - \nuo_{\tau}) (x, r) \bigg) \beta(\tilde{x}, \tilde{r})\di\rhoo_\tau^N(\tilde x,\tilde r) \di \tau 
\\
&
\leq \lim_{N\to \infty} M T {\rm Lip} (\theta)   \| \alpha\|_{L^\infty([0,T])} \|\beta\|_{L^\infty(B_{\mathcal{R}}(0))} \sup_{\tau \in [0, T]} W_{1} (\nuo_{\tau}^{N}, \nuo_{\tau})  = 0\,. \nonumber
\end{align}
Since $\nuo \in C([0, T]; \mathcal{P}_{1} (\R^{2d}))$ and $\rhoo^{N}$ converges narrow to~$\rhoo$, passing to the limit in~\eqref{pr14tris} and using \eqref{pr14quarter}, we obtain \eqref{pr14bis}. Consequently, we deduce that
\begin{equation*}
\lim_{N\to +\infty} I_2^N = 0,
\end{equation*}
which in turn implies, together with \eqref{pr12} and \eqref{pr13}, that 
\begin{eqnarray*}
&&\lim_{N\to +\infty} \int_0^t \int_{\R^{2d}}\int_{\R^{2d}} \langle \nabla_r\varphi(\tau,x,r), \nabla_{\psi} h(\tilde x,\Psi_\tau^N)(x)\rangle\langle \tilde r, \di\rhoo_\tau^N(\tilde x,\tilde r)\rangle \di\nuo^N_\tau(x,r)\di \tau \\
&& = \int_0^t \int_{\R^{2d}}\int_{\R^{2d}} \langle \nabla_r\varphi(\tau,x,r), \nabla_{\psi} h(\tilde x,\Psi_\tau)(x)\rangle\langle \tilde r, \di\rhoo_\tau(\tilde x,\tilde r)\rangle \di\nuo_\tau(x,r)\di \tau.
\end{eqnarray*}
Repeating the same argument for the other integrals on the right-hand side of \eqref{pr11}, we conclude that $\nuo$ is a distributional solution to \eqref{syst-tot-inf}. \\
Finally, in order to get the maximality condition, we start by taking $\omega\in \mathrm{Lip}(\R^d; K)$. We define $\uv^N$ with components $u_i(t):=\omega(\xo_i(t))$ for $i=1,\dots,N$ and $t\in [0,T]$. Applying Proposition \ref{vecchio2} in the inequality below and using the definition of $\rhoo^N$ and $\nuo^N$ (see \eqref{def-rhoo} and \eqref{def-nuo} respectively), we obtain for any $[t_1,t_2]\subseteq [0,T]$
\begin{eqnarray}
\label{pr15}
%\int_{t_1}^{t_2}\mathcal{H}\left(\nuo^N_t,\frac{\di \rhoo^N}{\di \nuo^N}\right) \di t \stackrel{\eqref{ham-inf},\eqref{ham-N}} 
&& \int_{t_1}^{t_2} \mathcal{H}_N(\xo^N(t),\ro^N(t),\uo^N(t)) \di t  \geq \int_{t_1}^{t_2} \mathcal{H}_N(\xo^N(t),\ro^N(t),\uv^N(t))\di t  \nonumber
 \\
 &&
 \stackrel{\eqref{ham-inf},\eqref{ham-N}}= \int_{t_1}^{t_2}\mathcal{H}\left(\nuo^N_t,\omega\right) \di t.
\end{eqnarray} 
We want to pass to the limit in \eqref{pr15} as $N\to +\infty$. First we focus on the right-hand side. By definition of $\mathcal{H}$ (see \eqref{ham-inf}), since $\nuo^N$ converges to $\nuo$ in $C([0,T];\mathcal{P}_1(\R^{2d}))$ and $\supp(\nuo^N)\subseteq [0,T]\times B_{\mathcal{R}}(0)$ (see Lemma \ref{lemma1}), using the continuity assumption on $L$ and the fact that $\phi(\omega)\in C_b(\R^d)$ it follows that
\begin{equation}
\label{pr17}
\lim_{N\to+\infty} \int_{t_1}^{t_2} L(\pi^1_{\#}\nuo_t^N) \di t= \int_{t_1}^{t_2} L(\pi^1_{\#}\nuo_t) \di t 
\end{equation}
and
\begin{equation}
\label{pr18}
\lim_{N\to +\infty}\int_{t_1}^{t_2} \int_{\R^{2d}} \phi(\omega(x))\di\nuo_t^N(x,r)\di t = \int_{t_1}^{t_2} \int_{\R^{2d}} \phi(\omega(x))\di\nuo_t(x,r)\di t.
\end{equation}
Moreover, arguing as done to estimate \eqref{pr12} (which is possible since $\supp(\nuo)$ is compact and $\omega$ is continuous) we deduce
\begin{eqnarray}
\label{pr16}
&&\lim_{N\to +\infty} \int_{t_1}^{t_2}\int_{\R^{2d}} \langle r, v(x,\pi_{\#}^1\nuo^N_t) + h(x,\pi_{\#}^1\nuo_t^N)\omega(x)\rangle \di \nuo^N_t(x,r)\di t \nonumber \\
&&=\int_{t_1}^{t_2}\int_{\R^{2d}} \langle r, v(x,\pi_{\#}^1\nuo_t) + h(x,\pi_{\#}^1\nuo_t)\omega(x)\rangle \di \nuo_t(x,r)\di t. 
\end{eqnarray}
In the same way, leveraging on the narrow convergence of $\nuo^N$ to $\nuo$ and of $\rhoo^N$ to $\wvo\nuo$ given by Lemma \ref{lemma3}, we have
\begin{eqnarray}
\label{pr19}
&&\lim_{N\to +\infty} \int_{t_1}^{t_2} \bigg(\int_{\R^{2d}} \langle r, v(x,\pi_{\#}^1\nuo^N_t)\rangle \di \nuo^N_t(x,r) + \langle r, h(x,\pi_{\#}^1\nuo_t^N)\di\rhoo_t^N(x,r) \rangle\bigg)\di t \nonumber \\
&&=\int_{t_1}^{t_2}\int_{\R^{2d}} \langle r, v(x,\pi_{\#}^1\nuo_t) + h(x,\pi_{\#}^1\nuo_t)\wvo(t,x)\rangle \di \nuo_t(x,r)\di t. 
\end{eqnarray}
Now, recalling the definition of $\mathcal{H}$ and of $\mathcal{H}_{N}$ (see \eqref{ham-inf} and \eqref{ham-N}), combining \eqref{pr16}, \eqref{pr17} and \eqref{pr18} for the last term on right-hand side of \eqref{pr15} and using \eqref{pr19}, \eqref{pr17} and \eqref{gamma-conv3} for the first term on the left-hand side of \eqref{pr15}, we infer that
\begin{align*}
\int_{t_1}^{t_2} \mathcal{H}\left(\nuo_t,\wvo(t,\cdot)\right) \di t & = \lim_{N\to +\infty}  \int_{t_1}^{t_2} \mathcal{H}_N(\xo^N(t),\ro^N(t),\uo^N(t)) \di t \nonumber \\
&\geq \lim_{N\to +\infty} \int_{t_1}^{t_2} \mathcal{H}\left(\nuo_t^N,\omega\right) \di t = \int_{t_1}^{t_2} \mathcal{H}\left(\nuo_t,\omega\right) \di t,
\end{align*}
for every $\omega\in \mathrm{Lip}(\R^d;K)$ and for every $t_1,t_2 \in [0,T]$. By arbitrariness of $t_1$ and $t_2$ in $[0,T]$ and applying the Lebesgue differentiation theorem we conclude that for every $\omega\in \mathrm{Lip}(\R^d;K)$
\begin{equation}
\label{pr20}
\mathcal{H}\left(\nuo_t,\wvo(t,\cdot)\right) \geq \mathcal{H}\left(\nuo_t,\omega\right) \qquad \text{for a.e. }t\in [0,T].
\end{equation}
Note that, by \eqref{ham-inf} and since $\nuo\in C([0,T];\mathcal{P}_1(\R^{2d}))$, every $t\in [0,T]$ is a Lebesgue point of $\mathcal{H}(\nuo_t,\omega)$. Therefore \eqref{pr20} holds at every Lebesgue point of $\mathcal{H}(\nuo_t,\wvo(t,\cdot))$, in particular such points depend only on $\wvo$. \\
Finally, by density of $\mathrm{Lip}(\R^d;K)$ in $L^1_{\pi^1_{\#}\nuo_t}(\R^d;K)$ and recalling that $K$ is compact, we deduce that \eqref{pr20} holds for every $\omega\in L^1_{\pi^1_{\#}\nuo_t}(\R^d;K)$, which in turn implies \eqref{cond-inf}.
\end{proof}

\section{Generalization to optimal control problems with convex state space}
\label{s:convexcase}

In this section we briefly discuss how to generalize the setting presented above in~$\R^{d}$ to the convex constrained framework introduced in~\cite{AFMS, MS2020} for modelling multi-agent multi-label systems. In this case, the state variable is a pair $(x, \lambda) \in C := \mathbb{R}^{d} \times \mathcal{P} (U)$, where~$\mathcal{P}(U)$ is the space of probability measure on a finite set of pure strategies $U$. The agents' state is therefore described by their position~$x$ and by their mixed strategy~$\lambda$.  We denote by $c = (x, \lambda)$ the generic element of~$C$.  In the following two examples, inspired by~\cite[Section~5]{ADS} and~\cite{ADEMS, MS2020}, we point out how the analysis performed above can be adapted by replacing usual gradients in~$\R^{d}$ with the notion of $C$-differentiability, reported here.

\begin{definition}
\label{d:cdiff}
Let $(E, \| \cdot\|_{E})$ and $(F , \| \cdot \|_{F})$ be two normed spaces, $C \subseteq E$ a closed convex subset of~$E$, and $f \colon C \to F$. We say that $f$ is~$C$-differentiable in $c \in C$ if there exists a linear operator ${\rm D}_{c} f \in \mathcal{L} (E_{C}; F)$ such that
\begin{displaymath}
\lim_{C\ni c' \to c}\, \frac{f(c') - f(c) - {\rm D}_{c} f [c' - c]}{\| c' - c\|_{E}} = 0\,.
\end{displaymath}
\end{definition}

\subsection{Control of multi-population systems}
\label{s:example-1}

In multi-population systems, the space dynamics of each agent is coupled with a transition process for the label $\lambda$, describing how agents may switch the population they belong to. Such process is modelled by means of reversible Markov chains on $n$ states (cf.~\cite{Maas, Mielke}). \\
For simplicity, we identify the set of labels $U$ with the canonical base of~$\R^{n}$, i.e.,  $U = \{ e_{1}, \ldots, e_{n}\}$, and endow $U$ with the distance
\begin{equation}
\label{e:metric-examples}
0= d_{U} (e_{i}, e_{i}) \quad \text{for $i = 1, \ldots, n$}, \qquad  1=d_{U}(e_{i} , e_{j}) \quad \text{for $i \neq j$}.
\end{equation}
The space of probability measures $\cP(U)$ is identified with the closed $(n-1)$-simplex 
\begin{align*}
\Lambda_{n}:= \bigg\{ \lambda = (\lambda_{1}, \ldots, \lambda_{n}) \in \R^{n}: \lambda_{i} \geq 0\,, \ \sum_{i=1}^{n} \lambda_{i} = 1 \bigg\}\,.
\end{align*}
The state space is represented by the convex subset $C= \R^{d} \times \mathcal{P}(U) \sim \R^{d} \times \Lambda_{n}$ of $E= \R^{d} \times \mathcal{F} (U)$, where $\mathcal{F}(U)$ is the {\em Aerens-Eelles} space (see~\cite[Section~2.1]{AFMS} and \cite{AP}).  We notice that
\begin{displaymath}
E_{C} = \overline{ \R(C - C)} =\R^{d} \times  \{ \mu \in \mathcal{M} (U): \, \mu(U) = 0\}\,.
\end{displaymath}
Since $\mathcal{P} (U)$ is identified with~$\Lambda_{n}$,~$E_{C}$ may be represented by $\R^{d} \times \R^{n}_{0}$, where
\begin{displaymath}
\R^{n}_{0} := \bigg\{ \xi \in \R^{n}: \, \sum_{i=1}^{n} \xi_{i} = 0 \bigg\}\,.
\end{displaymath}
In particular, we notice that $E$ is a finite dimensional space, and thus a separable, reflexive, and locally compact Banach space. 

Given a compact, convex set of admissible controls~$K\subseteq \R^{d}$ with $0 \in K$, we consider $\K:=L^1([0,T];K)$ and set up the control problem
\begin{equation}
\label{costoN-ex}
\min_{\bm u^N\in\K^N}\left\{ \int_0^T L(\psi_t^N) \di t + \int_0^T \frac{1}{N}\sum_{i=1}^N \phi(u_i(t))\di t \right\}
\end{equation}
where $\psi^{N}_{t} := \frac{1}{N} \sum_{i=1}^{N} \delta_{(x_{i} (t) , \lambda_{i} (t))}$ and $c_{i}= (x_{i}, \lambda_{i})$ satisfies
\begin{equation}
\label{systN-ex}
\begin{cases}
\dis \dd x_i(t)= v(c_i(t),\psi_t^N) + h(c_i(t),\psi_t^N)u_i(t) & \text{ in }(0,T],  \\[2mm]
\dis \dd \lambda_{i} (t) = \mathcal{T} (c_{i} (t) , \psi^{N}_{t}) &  \text{ in }(0,T],\\[2mm]
x_i (0)=\xo_{0,i}^N, \qquad \lambda_{i} (0) = \boldsymbol{\lambda}_{0, i}^{N}
\end{cases}
\quad\text{for }i=1,\dots, N,
\end{equation}
where $\mathcal{T}(c, \psi) := \mathcal{Q} (x, \psi) \lambda$ for a matrix-valued map $\mathcal{Q}\colon \R^{d} \times \mathcal{P}_{1} (C)  \to \R^{n \times n}$ satisfying the following:
\begin{itemize}
\item [$(\mathcal{Q}_{0})$] for every $(x, \psi) \in \R^{d} \times \mathcal{P}_{1}(C)$ and every $i, j = 1, \ldots, n$, $\mathcal{Q}_{ij}(x, \psi) \geq 0$ for~$i \neq j$, and $\mathcal{Q}_{ii}(x, \psi) = - \sum_{j \neq i} \mathcal{Q}_{j i} (x, \psi)$; 

\item[$(\mathcal{Q}_{1})$] for every $(x, \psi) \in \R^{d} \times \mathcal{P}_{1}(C)$, $\mathcal{Q} ( x, \psi)$ is \emph{reversible}, that is, there exists a unique $\sigma= \sigma (x, \psi) \in \Lambda_{n}$ such that
\begin{displaymath}
\mathcal{Q}_{ij} ( x, \psi ) \sigma_{j} = \mathcal{Q}_{ji} (x, \psi) \sigma_{i} \qquad \text{for every $i, j = 1, \ldots, n$}\,,
\end{displaymath}
%\item[$(\mathcal{Q}_{2})$] $\Qq$ is locally Lipschitz, that is, for every $R>0$ there exists $L_{\Qq, R}>0$ such that for every $x_{1}, x_{2} \in B_{R}$ and every $\Psi_{1}, \Psi_{2} \in \Pp (\B^{Y}_{R})$
%\begin{displaymath}
%| \Qq( x_{1}, \Psi_{1}) - \Qq (x_{2}, \Psi_{2}) | \leq L_{\Qq, R} \big( | x_{1} - x_{2} | + W_{1} (\Psi_{1}, \Psi_{2}) \big)\,;
%\end{displaymath}
%
%\item[$(\Qq_{3})$] there exists $M_{\Qq} > 0$ such that for every $x \in \R^{d}$ and every $\Psi \in \Pp_{1}(Y)$
%\begin{displaymath}
%| \Qq ( x, \Psi) | \leq M_{\Qq} \big( 1 + |x| + m_{1} ( \Psi ) \big) \,.
%\end{displaymath}
\end{itemize}
together with local Lipschitz continuity, linear growth and differentiability conditions similar to {\bf (Hv)} and {\bf (Hh)}. We refer to \cite[Section~5.1]{ADS} for explicit examples of the fields $v_{\psi}$,~$h$, $\mathcal{Q}$, $L$, and~$\phi$ above, together with a discussion concerning their $C$-differentiability and Wasserstein differentials in the case $n=2$, which can be easily extended to any $n \geq 2$. In particular, continuity of $C$- and Wasserstein differentials is discussed, which is part of the assumptions~{\bf (Hv)}.

Existence of optimal controls for~\eqref{costoN-ex}--\eqref{systN-ex} has been studied in~\cite{AAMS}, together with the variational convergence for a diverging number of particles $N$ to the following mean-field optimal control problem 
\begin{equation}
\label{control-convex-sec5}
\min_{\omega\in L^1_{\psi}([0,T]\times C ; K)}\left\{ \int_0^T L(\psi_t) \di t + \int_0^T \int_{\R^d} \phi(\omega(t,c))\di \psi_t(c) \di t \right\}
\end{equation}
subjected to
\begin{equation}
\label{control-convex-sec52}
\begin{cases}
\dis \dd \psi_t = -\mathrm{div} \Big ( (v(c,\psi_t), \mathcal{T} (x, \psi_{t}) ) \psi  + ( h(c,\psi_t)\omega(t,c)), 0) \psi_t \Big) & \text{ in }(0,T],  \\
\psi_0=\hat{\Psi}_0\,,
\end{cases}
\end{equation}
for $\hat{\Psi}_{0} \in \mathcal{P}_{c} (C)$ limit of $\psi^{N}_{0}$ in the $1$-Wasserstein distance. We refer to~\cite[Theorem~2 and Corollary~1]{AAMS} for the precise statement, in the spirit of Proposition~\ref{vecchio}. Optimality conditions in Pontryagin form for~\eqref{costoN-ex}--\eqref{systN-ex} in the case of smooth optimal controls have been obtained in~\cite[Theorem~3.5]{ADS}, relying on the notion of $C$-differentiability of Definition~\ref{d:cdiff}. Hinging on the finite dimensional nature of the state space~$C \subseteq E$, the results contained in Theorem~\ref{mainres} can be repeated verbatim for the control problem~\eqref{control-convex-sec5}--\eqref{control-convex-sec52}, replacing space gradients with~$C$-differentials (cf.~Definition~\ref{d:cdiff}). Notice that the curve ${\bm \nu}$ belongs to $\mathrm{Lip}([0, T]; \mathcal{P}_1( C \times E_{C}^{*}))$ and ${\bm \rho} \in \mathcal{M} ([0,  T]\times C \times E^{*}_{C}; \R^{d})$ in this setting, where $E^{*}_{C}$ denotes the dual space to~$E_{C}$.

\subsection{Entropy regularized replicator dynamics}
\label{e:example-2}

The second class of examples we consider in multi-label systems is that entropy regularised replicator dynamics, inspired by~\cite{ADEMS, BFS} and~\cite[Section~5.2]{ADS} (see also \cite{TrClHh,ChSo,MoOr}).  As in Section~\ref{s:example-1}, we consider the set of labels $U=\{e_{1}, \ldots, e_{n}\} \subseteq \R^{n}$ endowed with the metric~\eqref{e:metric-examples}. We fix a probability measure $\eta \in \mathcal{P}(U)$ with $\supp(\eta) = U$, and $p \in (1, +\infty)$, and define $E := \R^{d} \times L^{p}_{\eta} (U)$, where
\begin{displaymath}
L^{p}_{\eta} (U) := \bigg\{ \lambda \colon U \to \R: \, \int_{U} | \lambda(u)|^{p} \, \di \eta(u) <+\infty \bigg\}.
\end{displaymath}
The space~$E$ is endowed with the norm $\| \cdot\|_{E} = | \cdot| + \| \cdot\|_{p}$, where $\| \cdot\|_{p}$ denotes the $L^{p}$-norm of $L^{p}_{\eta}(U)$. Since $U$ is finite, $E$ is a finite dimensional Banach space, and thus separable, reflexive, and locally compact. We further remark that, being $\supp(\eta) = U$, $\eta$ is a sum of deltas and $\| \cdot\|_{p}$ is a weighted version of the standard $p$-norm of~$\R^{n}$. 

For $0 < r < R <+\infty$ we set
\begin{displaymath}
C_{r, R} := \R^{d} \times \bigg\{ \lambda \in L^{p}_{\eta} (U): \, r \leq \lambda (u) \leq R \ \text{for $\eta$-a.e.~$u \in U$} \, \text{and} \, \int_U|\lambda(u)|^p\di \eta(u) = 1 \bigg\}.
\end{displaymath}
In particular, $C_{r, R}$ is a convex and closed subset of~$E$. We denote by $c =(x, \lambda)$ the elements of~$C_{r, R}$ and consider the set of controls $\K:=L^1([0,T];K)$ as in Section~\ref{s:example-1}. For every $N \in \mathbb{N}$, let us consider the finite particle control problem~\eqref{costoN-ex}--\eqref{systN-ex}, where $\mathcal{T}(c, \psi) := \mathcal{S} (c, \psi) + \varepsilon \mathcal{R}(\lambda)$ for $\varepsilon >0$, where for every $\psi \in \mathcal{P} (C_{r, R})$ we have set
\begin{align*}
&
 \mathcal{S} (c, \psi) := \bigg( \int_{C_{r,R}} J(x, \cdot , x') \, \di \psi(x', \lambda') - \int_{U} \int_{C_{r,R}}  J(x, u' , x') \, \lambda(u')\, \di \psi(x', \lambda')  \, \di \eta(u')\bigg) \lambda, \\
 & \mathcal{R} (\lambda) := \bigg( \int_{U} \lambda(u)\,\log(\lambda (u)) \,\de\eta(u) - \log (\lambda) \bigg) \lambda\,.
\end{align*}
for a Lipschitz continuous payoff function $J \colon \R^{d} \times U \times \R^{d} \to \R$. Well-posedness of~\eqref{systN-ex} in $C_{r, R}$ for a given set of controls~${\bm u^N}$ is contained in~\cite{ADEMS}. The $C$- and Wasserstein differentiability of~$\mathcal{S}$ and~$\mathcal{R}$ have been discussed in~\cite[Section~5.2]{ADS} under differentiability assumptions on the payoff function~$J$ with respect to $x$ and $x'$. The continuity of such differentials follows from the continuity of $\nabla_{x} J$ and $\nabla_{x'} J$. Arguing as in Proposition~\ref{vecchio}, also in this case we may recover the mean-field optimal control problem~\eqref{control-convex-sec5}--\eqref{control-convex-sec52} as variational limit of~\eqref{costoN-ex}--\eqref{systN-ex}. Finally, the optimality conditions and the convergences discussed in Theorem~\ref{mainres} can be deduced, still relying on the local compactness of~$E$ and~$E^{*}_{C}$ (recall they are both finite dimensional spaces).

\section*{Acknowledgements}
The work of S. Almi was funded by the FWF Austrian Science Fund through the Project 10.55776/P35359 and by the University of Naples Federico II through FRA Project "ReSinApas".
\par
R. Durastanti has been supported by the Italian Ministry of University and Research under PON ``Ricerca e Innovazione'' 2014-2020 (PON R\&I, D.M. 1062/21) - AZIONE IV.6 ``Contratti di Ricerca su tematiche Green'' CUP E65F21003200003, and, his work has been carried out in collaboration with CRdC Tecnologie Scarl as part of the "Embodied Social Experiences in Hybrid Shared Spaces (SHARESPACE)" project - http://sharespa\\ce.eu funded by the European Union under Horizon Europe, grant number 101092889.
\par
The work of R. Durastanti and F. Solombrino has been also supported by Gruppo Nazionale per l'Analisi Matematica, la Probabilit\`a e le loro Applicazioni (GNAMPA-INdAM, Project 2024 ``Problemi di controllo ottimo nello spazio di Wasserstein delle misure definite su spazi di Banach'', CUP E53C23001670001). The work of S.~Almi is supported by  Gruppo Nazionale per l'Analisi Matematica, la Probabilit\`a e le loro Applicazioni (GNAMPA-INdAM, Project 2025: DISCOVERIES - Difetti e Interfacce in Sistemi Continui: un'Ottica Variazionale in Elasticit\`a con Risultati Innovativi ed Efficaci Sviluppi).
\par
The work of F. Solombrino and S. Almi is part of the MUR - PRIN 2022, project Variational Analysis of Complex Systems in Materials Science, Physics and
Biology, No. 2022HKBF5C, funded by European Union NextGenerationEU.

\end{document}